\documentclass[11pt]{amsart}
\usepackage{graphicx,amsfonts,amssymb,amsmath,amsthm,url,
  verbatim,amscd}
  \usepackage{color}
\usepackage[usenames,dvipsnames]{xcolor}
\usepackage[normalem]{ulem}
\usepackage{enumerate}
\usepackage{pdfsync}
\usepackage{booktabs}
\usepackage{subcaption, diagbox} 
\usepackage{fullpage}
\usepackage[hyperfootnotes=false, colorlinks, citecolor=RoyalBlue,
urlcolor=blue, linkcolor=blue ]{hyperref}
\usepackage{bbm}
\newcommand{\tst}{\textstyle}
\allowdisplaybreaks

\theoremstyle{plain}
\newtheorem*{theorem*}{Theorem}
\newtheorem{theorem}  {Theorem}    [section]
\newtheorem{lemma}      [theorem]{Lemma}
\newtheorem{corollary}  [theorem]{Corollary}
\newtheorem{proposition}[theorem]{Proposition}

\newtheorem{remark}  [theorem] {Remark}
\theoremstyle{definition}

\newtheorem{definition} [theorem]{Definition}

\DeclareMathOperator{\Cl}{Cl}
\newcommand{\AI}{{\mathcal{AI}}}
\renewcommand{\H}{\mathbb H}

\newcommand{\A}{{\mathbb A}}
\newcommand{\Q}{{\mathbb Q}}
\newcommand{\Z}{{\mathbb Z}}
\newcommand{\R}{{\mathbb R}}
\newcommand{\C}{{\mathbb C}}
\newcommand{\bs}{\backslash}

\newcommand{\p}{\mathfrak p}
\newcommand{\eps}{\epsilon}
\renewcommand{\a}{\alpha}
\renewcommand{\b}{\beta}

\newcommand{\OF}{{\mathfrak o}}
\newcommand{\GL}{{\rm GL}}

\newcommand{\st}{{\rm st}}
\newcommand{\SL}{{\rm SL}}
\newcommand{\SO}{{\rm SO}}

\newcommand{\GSp}{{\rm GSp}}

\newcommand{\Sp}{{\rm Sp}}

\newcommand{\Ad}{{\rm Ad}}
\newcommand{\sym}{{\rm sym}}

\newcommand{\vol}{{\rm vol}}

\newcommand{\Norm}{{\rm N}}
\newcommand{\Mat}{{M}}

\newcommand{\trace}{{\rm tr}}
\newcommand{\Tr}{{\rm Tr}}

\newcommand{\disc}{{\rm disc}}
\newcommand{\T}[1]{{}^t\!{{#1}}}

\renewcommand{\(}{\left(} \renewcommand{\)}{\right)}
\newcommand{\mat}[4]{\begin{bsmallmatrix}#1&#2\\#3&#4\end{bsmallmatrix}}

\newenvironment{bsmallmatrix}{\left[\begin{smallmatrix}}{\end{smallmatrix}\right]}

\def\Sc{\mathcal S}
\def\1{\mathbbm{1}}

\title{Bounds on Fourier coefficients and global sup-norms for Siegel cusp forms of degree 2}
\author{F\'elicien Comtat, Jolanta Marzec-Ballesteros, Abhishek Saha}
\address{Mathematical Institute, University of Bonn, D-53115 Bonn, Germany}
\email{comtat@math.uni-bonn.de}
\address{Faculty of Mathematics and Computer Science, Adam Mickiewicz University, 61-614 Pozna\'n, Poland}
\email{jmarzec@amu.edu.pl}
\address{School of Mathematical Sciences, Queen Mary University of London, London E1 4NS, UK}
\email{abhishek.saha@qmul.ac.uk}


\begin{document}

\begin{abstract}
Let $F$ be an $L^2$-normalized Siegel cusp form for $\Sp_4(\Z)$ of weight $k$ that is a Hecke eigenform and not a Saito--Kurokawa lift. Assuming the Generalized Riemann Hypothesis, we prove that its Fourier coefficients satisfy the bound $|a(F,S)| \ll_\eps \frac{k^{1/4+\eps} (4\pi)^k}{\Gamma(k)} c(S)^{-\frac12} \det(S)^{\frac{k-1}2+\eps}$ where $c(S)$ denotes the gcd of the entries of $S$, and that its global sup-norm satisfies the bound $\|(\det Y)^{\frac{k}2}F\|_\infty \ll_\epsilon k^{\frac54+\epsilon}.$ The former result depends on new bounds that we establish for the relevant local integrals appearing in the refined global Gan-Gross-Prasad conjecture (which is now a theorem due to Furusawa and Morimoto) for Bessel periods.
\end{abstract}

\maketitle

\section{Introduction}\label{s:intro}
The problem of bounding the sup-norms of $L^2$-normalized cuspidal automorphic forms as one or more of their underlying parameters tend to infinity is interesting from several points of view and has been the topic of many recent works, see e.g.  \cite{iwan-sar-95, harcos-templier-2,   marshall, blomer-pohl, saha-sup-level-hybrid, blomer-harcos-milicevic-maga, brumley-templier, sup-norm-minimal-compact, HS19, KNS22} and the references therein.
A basic case of this problem concerns upper-bounds for the global sup-norms of holomorphic cusp forms $f$ of weight $k$ for $\SL_2(\Z)$ as $k \rightarrow \infty$. Assuming that $f$ is an eigenform, Xia \cite{xia} proved the bound \begin{equation}\label{e:xiabd}\|y^{k/2}f\|_\infty \ll_\eps k^{\frac14 +\eps}\|f\|_2.\end{equation}  The exponent $1/4$ here is optimal, as one can prove a lower bound of similar strength\footnote{The lower bound of $k^{\frac14 - \eps}$ is obtained high in the cusp. A variant of the sup-norm problem focusses not on global bounds but on bounds over a fixed compact set $\Omega$ where it is an open problem to improve upon the exponent $1/4$. Indeed, the much stronger upper bound $\|y^{k/2}f|_{\Omega}\|_\infty \ll_\eps k^\eps$  is expected to hold.}. The proof uses the Fourier expansion $f(z) = \sum_{n>0}a(f, n) e^{2 \pi i nz}$ and crucially relies on Deligne's bound \begin{equation}\label{e:delbd}|a(f, n)| \le d(n)n^{\frac{k-1}{2}} |a(f, 1)|\end{equation} for the Fourier coefficients of $f$ as well as the bound  $\frac{|a(f, 1)|}{\|f\|_2} \ll_\eps k^\eps$ which follows from the non-existence of Landau--Siegel zeroes \cite{HL94} for the symmetric-square $L$-function attached to $f$. Given these deep facts, the deduction of \eqref{e:xiabd} from the Fourier expansion is fairly direct. The fact that the bound \eqref{e:xiabd} is essentially best possible reflects the special behaviour of the exponential function as a degenerate Whittaker function. 

In this paper we are interested in a rank 2 analogue of Xia's result. Namely, let $\H_{2}$ denote the Siegel upper-half space of degree $2$ and let $S_k(\Gamma)$
be the space of holomorphic Siegel cusp forms
of weight $k$ transforming with respect to the subgroup $\Gamma = \Sp_{4}(\Z) \subset \Sp_{4}(\R)$.
As in the rank 1 case considered by Xia, one may hope to exploit the Fourier expansion and obtain strong bounds on the sup-norm in this setting.
Recall that the Fourier expansion of $F \in S_k(\Gamma)$ takes the form
\begin{equation}\label{e:fouriercoeffsiegelform}\tst
 F(Z)=\sum_{S\in \Lambda_2} a(F, S)e^{2\pi i{\rm Tr}(SZ)},\qquad Z \in \H_2,
\end{equation}
where
$$ \Lambda_2 = \left\{\mat{a}{b/2}{b/2}{c}:\qquad a,b,c\in\Z, \qquad a>0, \qquad  d:=b^2 - 4ac <0\right\}.
$$
For $S =\mat{a}{b/2}{b/2}{c}\in \Lambda_2$, we define its discriminant $\disc(S)=-4\det(S)=b^2-4ac$ and its content $c(S) = \gcd(a, b, c)$. Note that $c(S)^2$ divides $\disc(S)$. If $d = \disc(S)$ is a fundamental discriminant\footnote{Recall that an integer $n$ is a fundamental discriminant if \emph{either} $n$ is a squarefree integer congruent to 1 modulo 4 \emph{or} $n = 4m$ where $m$ is a squarefree integer congruent to 2 or 3 modulo 4.}, then $S$ is called \emph{fundamental} in which case clearly $c(S)=1$.

The Fourier coefficients of Hecke eigenforms in  $S_k(\Gamma)$  are mysterious and poorly understood objects. Unlike in the rank 1 case, they contain much more information than just the Hecke eigenvalues and are closely related to central values of $L$-functions. However, there exist special forms in $S_k(\Gamma)$
known as the \emph{Saito-Kurokawa lifts} for which the Fourier coefficients are relatively better understood. In fact, the Fourier coefficients of a Saito-Kurokawa lift $F$ can be
explicitly written in terms of the Fourier coefficients of a classical half-integral weight form $g \in S_{k-\frac12}(\Gamma_0(4))$ and there exists a simple relation between the Petersson norms of $F$ and $g$. Using these facts, Blomer \cite[Sec. 4]{Blomer}  observed that  if a Hecke eigenform $F$ is a Saito--Kurokawa lift, then  under the Generalized Lindel\"of hypothesis (GLH)  one has the following bound\footnote{See also \cite[Theorem 5.1.6]{ps22} for an extension to the case of Saito--Kurokawa lifts with square-free level, where the dependence of the implied constant on $F$ is not made explicit.} on the Fourier coefficients \begin{equation}\label{e:blomerSKfourierbd} \frac{|a(F,S)|}{\|F\|_2} \ll_\eps \frac{k^{1/4+\eps} (4\pi)^k}{\Gamma(k)} c(S)^{\frac12} \det(S)^{\frac{k}{2} - \frac{3}{4} + \eps}\end{equation} where the Petersson norm $\|F\|_2$  is defined via
$\|F\|_2^2 = \langle F, F \rangle = \int\limits_{\Gamma \bs \H_2} |F(Z)|^2 (\det Y)^{k-3} dX dY.$
Using the bound \eqref{e:blomerSKfourierbd}, Blomer obtained under GLH the following essentially optimal bound \cite[Theorem 2]{Blomer} on the sup-norm of a Hecke eigenform $F\in S_k(\Gamma)$ that is a Saito--Kurokawa lift: \begin{equation}\label{e:blomersksupbd}\|(\det Y)^{\frac{k}2}F\|_\infty \ll_\epsilon k^{\frac34+\epsilon} \|F\|_2.\end{equation}

\medskip
 The main obstacle to generalizing \eqref{e:blomersksupbd} to non-Saito--Kurokawa lifts and  obtaining a global sup-norm bound for any Hecke eigenform $F \in S_k(\Gamma)$ lies in obtaining a bound similar to \eqref{e:blomerSKfourierbd} for non-lifts. A key step for fundamental $S$ was taken in \cite{DPSS15} where weighted averages of fundamental Fourier coefficents for such $F$ were related via the refined  Gan--Gross--Prasad (GGP) period conjecture for $(\SO_5, \SO_2)$ (which is now a theorem due to Furusawa and Morimoto \cite{FM22}) to values of higher degree $L$-functions. Using this, a bound under GRH for the fundamental Fourier coefficients  was proved in \cite[Prop 3.17]{DPSS15}; see also \cite[Theorem C]{JLS20} for a related result which saves an additional power of $\log(\det(S))$ but where the implied constant depends on $F$.

 A main achievement of the present paper is to go \emph{beyond fundamental matrices} and obtain a uniform bound for $\frac{|a(F,S)|}{\|F\|_2}$ under GRH for all $S \in \Lambda_2$. To lay the groundwork for our theorem,  recall first that
$a(F, S)=a(F, \T{A}SA)
$
for $A\in\mathrm{SL}_2(\mathbb Z)$, i.e., the Fourier coefficient $a(F, S)$ depends only on the $\mathrm{SL}_2(\mathbb Z)$-equivalence class of $S$.  Let $D<0$ be congruent to 0 or 1 mod 4 and let $L$ be a positive integer. The set of $\SL_2(\Z)$-equivalence classes of matrices $S \in \Lambda_2$ whose content $c(S)$ equals $L$ and whose discriminant $\disc(S)$ equals $L^2D$ can be canonically identified with the class group $H_D$ of the imaginary quadratic order of discriminant $D$. We view the characters $\Lambda$ of the finite abelian group $H_D$ as Hecke characters of $K^\times \bs \A_K^\times$ where $K = \Q(\sqrt{D})$. Note that in the special case that $D$ is a fundamental discriminant, these are precisely the characters of the ideal class group. We prove the following theorem.
\begin{theorem}[see Theorem \ref{t:mainfouriergen}]\label{t:mainfourier}
Let $F \in S_k(\Gamma)$ be a Hecke eigenform with Fourier expansion given by \eqref{e:fouriercoeffsiegelform}. Assume that $F$ is not a Saito--Kurokawa lift and let $\pi$ be the automorphic representation generated by $F$. Let $D<0$ be an integer that is congruent to 0 or 1 mod 4 and let $L$ be a positive integer. Then
\begin{equation}\label{e:theoremintro}\sum_{\substack{S \in \Lambda_2 / \SL_2(\Z) \\ c(S)=L, \ \disc(S)=L^2 D}}|a(F,S)|^2 \ll_\eps \langle F, F \rangle \frac{(4 \pi)^{2k}}{\Gamma(2k-1)} L^{2k-3 + \eps}|D|^{k - \frac{3}{2} + \eps} \sum_{\Lambda\in\widehat{H_D}} \frac{L(1/2, \pi \times \AI(\Lambda))}{L(1, \pi, \Ad)}.\end{equation}
\end{theorem}
In the above theorem, we note that the $L$-values are known to be non-negative \cite[Theorem 1.1]{lapid03} and that the length of the sum on each side is  equal to $|H_D| \asymp |D|^{\frac12 + o(1)}$. Under GRH\footnote{Strictly speaking, all we need is GLH and a sufficiently strong zero-free region for $L(s, \pi, \Ad)$.}, we can bound $\frac{L(1/2, \pi \times \AI(\Lambda))}{L(1, \pi, \Ad)} \ll_\eps (kD)^\eps.$  Recalling that $4\det(S) = |\disc(S)| = L^2|D|$ and using the duplication formula for the Gamma function, we obtain under GRH the strong bound (Corollary \ref{FourierCoeffBound})
\begin{equation}\label{e:boundsquares}   \sum_{\substack{S \in \Lambda_2 / \SL_2(\Z)\\ c(S)=L, \ \disc(S)=L^2D}} \frac{|a(F,S)|^2}{\|F\|^2_2} \ll_\eps \frac{k^{1/2+\eps} (2\pi)^{2k}}{\Gamma(k)^2} L^{-1} |L^2D|^{k-1 + \eps} .
\end{equation}
The bound \eqref{e:boundsquares} may be viewed as an extension of the bound \eqref{e:blomerSKfourierbd} to non-Saito--Kurokawa lifts\footnote{A key point here is that for Saito--Kurokawa lifts $F$, $a(F,S)$ depends only on $c(S)$ and $\det(S)$, and so for the corresponding sum in the case of the Saito--Kurokawa lifts all terms on the left side are equal. This may not be true for non-Saito--Kurokawa lifts.}. We believe that \eqref{e:boundsquares} is optimal as far as the exponents on the right side are concerned. Assuming that most summands on the left side of \eqref{e:boundsquares} are of comparable size, one is led to the optimistic and far-reaching conjecture $\frac{|a(F,S)|}{\|F\|_2} \ll_\eps \frac{k^{1/4+\eps} (4\pi)^{k}}{\Gamma(k)}  \det(S)^{\frac{k}{2}-\frac{3}{4} + \eps}$ for individual Fourier coefficients, which refines the famous open conjecture of Resnikoff and Salda\~{n}a \cite{res-sald} (and whose proof seems well beyond reach even if one were to assume standard conjectures like GRH).
On the other hand, dropping all but one term from \eqref{e:boundsquares}, we obtain the following corollary.
\begin{corollary}Assume GRH. For a Hecke eigenform $F \in S_k(\Gamma)$ that is not a Saito--Kurokawa lift, we have for any $S \in \Lambda_2$ the bound \begin{equation}\label{e:c:intro}\frac{|a(F,S)|}{\|F\|_2} \ll_\eps \frac{k^{1/4+\eps} (4\pi)^k}{\Gamma(k)} c(S)^{-\frac12} \det(S)^{\frac{k-1}2+\eps}.\end{equation}
\end{corollary}
In contrast to \eqref{e:boundsquares}, the bound \eqref{e:c:intro} is not expected to be optimal (even though it assumes GRH) because we potentially lose a factor of $\frac{\det(S)^\frac14}{c(S)^\frac12}$ when we drop all the other terms.  In the body of the paper we do not assume GRH but instead assume that $\frac{L(1/2, \pi \times \AI(\Lambda))}{L(1, \pi, \Ad)}$ is bounded by a specific power of the analytic conductor and we write down analogous bounds to \eqref{e:boundsquares}, \eqref{e:c:intro} under this assumption; see Corollary \eqref{FourierCoeffBound}.

Using our bound \eqref{e:c:intro}, we obtain a global sup-norm bound for non Saito--Kurokawa lifts  $F \in S_k(\Gamma)$ under GRH.
\begin{theorem}[see Theorem \ref{t:mainsup}]\label{t:introsup}
    Assume GRH. Let $F \in S_k(\Gamma)$ be a Hecke eigenform  that is not a Saito--Kurokawa lift. Then
       $$\|(\det Y)^{\frac{k}2}F\|_\infty \ll_\epsilon k^{\frac54+\epsilon} \|F\|_2.$$
\end{theorem}
The reason the exponent $5/4$ in the above Theorem is weaker than the exponent $3/4$ proved by Blomer for Saito--Kurokawa lifts is the non-optimality of the bound \eqref{e:c:intro}. (We expect the true exponent for the sup-norm to be $3/4$ for non-lifts\footnote{This is a special case of Conjecture 1.1 of \cite{das-krishna} and is supported by heuristics of the Bergman kernel as well as the fact that one can prove a lower bound of $k^{3/4- \eps}$ for many non-lifts using a method similar to \cite[Theorem 1]{Blomer}.} as well).

The proof of Theorem \ref{t:introsup} follows from \eqref{e:c:intro} in a relatively straightforward manner via the Fourier expansion, similar to the analysis in \cite{Blomer}. We remark here that (as is often the case) the Fourier expansion gives better results near the cusp.
    Indeed, our proof shows that if $Y$ is Minkowski-reduced, then we have under GRH the stronger bound  $ (\det Y)^{\frac{k}2}|F(Z)|\ll_\epsilon (\det Y)^{-\frac14} \ k^{\frac54+\epsilon}\|F\|_2.$
    In particular, when $Z$ is ``high in the cusp", this provides extra power savings in $k$.
    Thus, one may hope to obtain improved sup-norm bounds if one can tackle the ``bulk" by different methods.

While we have focussed on the Fourier expansion approach towards the global sup-norm in this paper, alternative approaches toward Theorem \ref{t:introsup} exist.  For example, using a Bergman Kernel approach, Das--Krishna \cite{das-krishna} have proved the bound $\| (\det Y)^{\frac{k}2}F\|_\infty \ll_\epsilon k^{\frac94+\epsilon} \|F\|_2$. There is also an exciting new approach to the sup-norm problem via 4th moments and theta kernels introduced by Steiner et al \cite{steiner20, khayutinsteiner, KNS22} which, if implemented for Siegel cusp forms, could potentially lead to strong bounds.

We end this introduction with a few words about the proof of Theorem \ref{t:mainfourier}, which should be viewed as the central result of this paper. The reader may be tempted to try to derive Theorem \ref{t:mainfourier} (or the stronger Theorem \ref{t:mainfouriergen}) from its known special case \cite{DPSS15} for $D$ a fundamental discriminant, by  using  the Hecke relations between fundamental and non-fundamental coefficients. However, this approach appears not to work. To see why, suppose that we know the values of all the fundamental Fourier coefficients and all the Hecke eigenvalues of $F$. Using the Hecke relations which can be expressed compactly via Sugano's formula as in \cite[Theorem 2.10]{kst2} we can then evaluate sums like $\sum_{\substack{S \in \Lambda_2 / \SL_2(\Z)\\  \pi^{-1}(L^{-1} S) = c } }a(F,S)$ with $c(S)=L$, $\disc(S)=L^2 D$ and $c$ a fixed class in $H_d$ where $d$ is the fundamental discriminant attached to $D$, and $\pi: H_d \rightarrow H_D$ is the natural map. However,  it does not seem possible to tell apart the individual summands above, i.e., this method cannot separate two non-equivalent, non-fundamental coefficients $a(F, S_1)$ and $a(F, S_2)$ of equal content $L$ in the case that $L^{-1}S_1$ and $L^{-1}S_2$ are images of the \emph{same} fundamental coefficient class under the natural map from $H_d \rightarrow H_D$.

Therefore we develop a different method for proving Theorem \ref{t:mainfourier}. We build upon the explicit refinement of B\"ocherer's conjecture introduced in \cite{DPSS15} and incorporate characters $\Lambda$ of $H_D$ which in general correspond to \emph{ramified} Hecke characters of $K=\Q(\sqrt{d})$. (In \cite{DPSS15}  we had restricted ourselves to ideal class group characters, i.e., characters of $H_d$; these correspond to unramified Hecke characters of $K$.) Using the refined GGP conjecture proved in \cite{FM22}, we write the left hand side of \eqref{e:theoremintro} as a deformation of the right hand side, where each summand is multiplied by a product over primes $p|D$ of certain  local quantities related to local Bessel models for ramified characters. Bounding these purely local quantities constitute the technical heart of this paper. This may be viewed as a \emph{depth aspect} bound for the local Bessel functions as the conductor of $\Lambda$ becomes large at the primes dividing $D$. For the main local statements, we refer the reader to Propositions \ref{p:localboundbesselsugano} and \ref{p:localboundJp}. In contrast to the similar local quantities for unramified characters computed in \cite{DPSS15}, the local quantities associated to ramified characters are not easy to compute explicitly, so we bound them via a soft approach where we exploit the volume of the support of the integrals. We refer the reader to Section \ref{s:proofprop2.2} for details of the argument. It would be of interest to compute these integrals exactly, as this would allow us to replace \eqref{e:theoremintro} by an exact identity.
\subsection*{Notation}We use the notation
$A \ll_{x,y,\ldots} B$
to signify that there exists
a positive constant $C$, depending at most upon $x,y,z$,
so that
$|A| \leq C |B|$. If the subscripts $x,y,\ldots$ are omitted, it means the constant is absolute. We write $A \asymp B$ to mean $A \ll B \ll A.$
 The symbol $\varepsilon$ will denote a small positive quantity.

 We say that an integer $d$ is a fundamental discriminant if $\Q(\sqrt{d})$ is a quadratic field whose discriminant is equal to $d$. For a fundamental discriminant $d$, we let $\chi_d$ be the associated quadratic Dirichlet character.
We use $\A$ to denote the ring of adeles over $\Q$ and for a number field $F$ we let $\A_F$ denote the ring of adeles over $F$. All $L$-functions in this paper will denote the finite part of the $L$-function (i.e., without the archimedean factors), so that for an automorphic representation $\pi$ of $\GL_n(\A)$, we have $L(s, \pi) = \prod_{p<\infty} L(s, \pi_p)$. All $L$-functions will be normalized to take $s \mapsto 1-s$. For a finite set of places $S$ we denote $L^S(s, \pi)=\prod_{p \notin S}L(s, \pi_p)$.

 For a commutative ring $R$, we define
\[
 \GSp_4(R)=\{g\in\GL_4(R):\:^tgJg=\mu(g)J,\:\mu(g)\in R^\times\},\qquad J=\left[\begin{smallmatrix}&&1\\&&&1\\-1\\&-1\end{smallmatrix}\right]
\]
Here, $\mu$ is called the similitude character. Let $\Sp_4(R) = \{g \in \GSp_4(R) : \mu(g) = 1\}$. We occasionally use $G$ to denote $\GSp_4$.

We let $M_2(R)$ denote the ring of 2 by 2 matrices over $R$, and let $M^\sym_2(R)$ be the additive subgroup of symmetric matrices.

Over the real numbers, we have the identity component $G(\R)^+:=\{g\in\GSp_4(\R):\mu(g)>0\}$.  Let $\H_2$ be the Siegel upper half space of degree $2$, i.e., the space $\H_2$ consists of the symmetric, complex $2\times 2$-matrices with positive definite imaginary parts. The group $G(\R)^+$ acts on $\H_2$ via $g\langle Z \rangle = (AZ+B)(CZ+D)^{-1}$ for $g = \mat{A}{B}{C}{D}$. We define $j(g, Z) = \det(CZ+D)$ for $g = \mat{A}{B}{C}{D} \in G(\R)^+$.

For a finite abelian group $H$ we let $\widehat{H}$ denote its group of characters.
\subsection*{Acknowledgments}The first and third-named authors acknowledge the support of the Engineering and Physical Sciences Research Council (grant number EP/W522508/1) and the Leverhulme Trust   (research project grant RPG-2018-401). The second-named author would like to thank for hospitality received at Kazimierz Wielki University where she wrote this work. We thank the anonymous referee for helpful comments which have improved this paper.
\section{Local calculations}\label{s:local}
In this section, which is purely local, we will study the  \emph{Bessel function} and \emph{Bessel integral} associated to a spherical vector in an irreducible, tempered, principal series representation of $\GSp_4(F)$ where $F$ is a non-archimedean local field of characteristic zero. Our main local results, Proposition~\ref{p:localboundbesselsugano} and Proposition~\ref{p:localboundJp}, quantify the growth of these quantities along  diagonal matrices.
\subsection{Basic facts and definitions}
\subsubsection{Preliminaries}Throughout Section \ref{s:local}, $F$ will be a non-archimedean local field of characteristic zero. Let $\OF$ be the ring of integers of $F$, with maximal ideal $\p$ and uniformizer $\varpi$. Let $\mathbf{k}=\OF/\p$ be the residue class field, and $q$ its cardinality. For $x \in F$, let $v(x)$ be the normalized valuation and let $| x | = q^{-v(x)}$ denote the normalized absolute value of $x$, so that $v(\varpi) = 1$, $|\varpi| = q^{-1}$.  Let $\psi$ be a character of $F$ which is trivial on $\OF$ but non-trivial on $\p^{-1}$.

We use the Haar measure $dx$ on $F$ that assigns $\OF$ volume 1, and we use the Haar measure $d^\times x$ on $F^\times$ that assigns $\OF^\times$ volume 1. So we have $d^\times x= (1-q^{-1})^{-1} \frac{ dx}{|x|}.$

\subsubsection{The Bessel subgroup} Following \cite{Fu} and \cite{PS1} we introduce the following notations. Let $a,b,c\in F$ such that  $d := b^2-4ac \neq 0$. Let
\begin{equation}\label{Sdefeq}
 S = \mat{a}{b/2}{b/2}{c}, \quad \Delta = \mat{b/2}{c}{-a}{-b/2}
\end{equation}
and note that $d= \disc(S) = -4 \det(S)$ and $\Delta^2=d/4$. If $d$ is not a square in $F^\times$, then let $K=F(\sqrt{d})$ and note that $K$ is isomorphic to $F(\Delta)$ via
\begin{equation}\label{isoeqinert}
F(\Delta) \stackrel{\sim}{\longrightarrow} K, \qquad x + y \Delta\longmapsto x + y \frac{\sqrt{d}}{2}.
\end{equation}
If $d$ is a square in $F^{\times}$, then let $K = F\oplus F$ and note that $K$ is isomorphic to $F(\Delta)$ via
\begin{equation}\label{isoeqsplit}
F(\Delta) \stackrel{\sim}{\longrightarrow} K, \qquad x + y \Delta\longmapsto \left(x + y \frac{\sqrt{d}}{2}, x - y \frac{\sqrt{d}}{2}\right).
\end{equation}

We define \begin{equation}\label{TFdefeq}
 T(F)=\{g\in\GL_2(F):\:^tgSg=\det(g)S\}.
\end{equation}
One can check that $T(F)=F(\Delta)^\times$, so that $T(F)\cong K^\times$ via the isomorphisms  \eqref{isoeqinert}, \eqref{isoeqsplit} above. We define the Legendre symbol as
\begin{equation}\label{legendresymboldefeq}
 \Big(\frac K\p\Big)=\begin{cases}
                      -1&\text{if $K/F$ is an unramified field extension},\\
                      0&\text{if $K/F$ is a ramified field extension},\\
                      1&\text{if }K=F\oplus F.
                     \end{cases}
\end{equation}
These three cases are referred to as the \emph{inert case}, \emph{ramified case}, and \emph{split case}, respectively. If $K$ is a field, then let $\OF_K$ be its ring of integers and $\p_K$ be the maximal ideal of $\OF_K$. If $K = F \oplus F$, then let $\OF_K = \OF \oplus \OF$.

Throughout, we will make the following \emph{standard assumptions} (see Section 1 of \cite{PS-BSI}),
\begin{equation}\label{standardassumptions}
 \begin{minipage}{90ex}
  \begin{itemize}
   \item$a, b \in\OF$ and $c\in\OF^\times$.
   \item If $d\notin F^{\times2}$, then $d$ is a generator of the discriminant of $K/F$.
   \item If $d\in F^{\times2}$, then $d\in\OF^\times$.
  \end{itemize}
 \end{minipage}
\end{equation}
Under these assumptions, the group $T(\OF):=T(F)\cap\GL_2(\OF)$ is isomorphic to $\OF_K^\times$ via the isomorphism $T(F)\cong K^\times$. Note also that these assumptions imply that if we are in the split case ($d\in F^{\times2}$) then $\frac{b \pm \sqrt{d}}{2} \in \OF$.

We consider $T(F)$ a subgroup of $\GSp_4(F)$ via
\begin{equation}\label{TFembeddingeq}
 T(F)\ni g\longmapsto\mat{g}{}{}{\det(g)\,^tg^{-1}}\in\GSp_4(F).
\end{equation}
Let $N(F)$ be the unipotent radical of the Siegel parabolic subgroup, i.e.,
$$N(F)=\{\mat{1_2}{X}{}{1_2}\in\GSp_4(F):\:^tX=X\}
$$
and let $R(F)=T(F)N(F)$. We call $R(F)$ the \emph{Bessel subgroup} of $\GSp_4(F)$.
\subsubsection{Bessel models} \label{s:besselmodels}
For $S$ as in \eqref{Sdefeq}, we define a character $\theta$ of $N(F)$ by
\begin{equation}\label{thetaSsetupeq}
 \theta(\left[\begin{smallmatrix} 1 &&x&y \\ &1&y&z \\ &&1& \\ &&&1 \end{smallmatrix}\right]) = \psi(ax+by+cz) = \psi (\trace(S\mat{x}{y}{y}{z}))
\end{equation}
for $x,y,z \in F$.  It is easily verified that $\theta(tnt^{-1})=\theta(n)$ for $n\in N(F)$ and $t\in T(F)$.

Let $\Lambda$ be any character of $K^\times$ such that $\Lambda |_{F^\times} = 1$. We identify $\Lambda$ with a character of $T(F)$ using the isomorphism $T(F) \simeq K^\times$. The map
$tu\mapsto\Lambda(t)\theta(u)$ defines a character of $R(F)$. We denote
this character by $\Lambda\otimes\theta$. Let $\mathcal{S}(\Lambda,\theta)$ be the space of all locally constant functions $B:\:\GSp_4(F)\rightarrow\C$ with the \emph{Bessel transformation property}
\begin{equation}\label{Besseltransformationpropertyeq}
 B(rg)=(\Lambda\otimes\theta)(r)B(g)\qquad\text{for all $r\in R(F)$ and $g\in \GSp_4(F)$}.
\end{equation}

Consider an irreducible, admissible, unitarizable, tempered representation $(\pi,V_\pi)$ of trivial central character. If  $(\pi,V)$ is isomorphic to a subrepresentation of $\mathcal{S}(\Lambda,\theta)$, then this realization of $\pi$ is called a \emph{$(\Lambda,\theta)$-Bessel model}. It is known that such a model, if it exists, is unique; we denote it by $\mathcal{B}_{\Lambda,\theta}(\pi)$.
Since $\pi$ is unitary, let $\langle \ , \ \rangle$ denote a $\GSp_4(F)$-invariant inner product (unique up to scaling) on $V_\pi$. For a vector $v \in V_\pi$, the normalized matrix coefficient attached to  $v$ is the function
\begin{equation}\label{matrixcoeff}
 \Phi_{v}(g)=\frac{\langle\pi(g)v,v\rangle}{\langle v, v \rangle}.
\end{equation}
Define
\begin{equation}\label{Besseldefn}
 J_{\Lambda, \theta}(v) := \int\limits_{F^\times \backslash T(F)} \int\limits_{N(F)}^{\st} \Phi_{v}(tn) \Lambda^{-1}(t) \theta^{-1}(n)\,dn\,dt,
\end{equation}
where $\int\limits_{N(F)}^{\st} := \lim_{k\rightarrow \infty}\int\limits_{N(\p^{-k})}$ denotes the stable integral \cite[Sect. 3.1]{yifengliu}.
It can be shown that the representation $\pi$ has a $(\Lambda, \theta)$-Bessel model (i.e., $\mathcal{B}_{\Lambda,\theta}(\pi)$ exists) \emph{if and only if} there is a non-zero vector $v$ in the space of $\pi$ such that $J_{\Lambda, \theta}(v) \neq 0$, in which case $v$ is said to be a $(\Lambda, \theta)$-test vector for $\pi$. We will refer to $J_{\Lambda, \theta}(v)$ as the \emph{local Bessel integral} (of type $(\Lambda, \theta)$) associated to $v$.

Suppose that the representation $\pi$ has a $(\Lambda, \theta)$-Bessel model.   Fix a realization of $\pi$ in $\mathcal{B}_{\Lambda,\theta}(\pi)$ and for each vector $v \in V_\pi$ let $B_v \in \mathcal{B}_{\Lambda,\theta}(\pi)$ denote its image. From uniqueness arguments it is easy to show that there exists a non-zero $c$ that depends on $\Lambda, \theta, \pi$ and on the choice of realization such that for all $v \in V_\pi$ and all $g \in \GSp_4(F)$ we have $|B_v(g)|^2 =  c \ J_{\Lambda, \theta}(\pi(g)v)$.

\subsubsection{Subgroups and characters of $T(\OF)$}\label{s:levelstructure}
Recall that $$T(\OF):=T(F)\cap\GL_2(\OF),$$ and that $T(\OF)$ is isomorphic to $\OF_K^\times$ via the isomorphisms \eqref{isoeqinert}, \eqref{isoeqsplit}. We define the subgroup $U_T(m) \subset T(\OF)$ via $$U_T(0) = T(\OF),$$ and for $m \ge 1$, $$U_T(m) = \{g \in T(\OF): g= \mat{\lambda}{}{}{\lambda} \bmod{\p^{m}},\,\text{ for some } \lambda\in\OF^\times \}.$$  Let $$\Delta_0:= \mat{0}{c}{-a}{-b}$$ so that $\Delta_0$ corresponds (under \eqref{isoeqinert}, \eqref{isoeqsplit} respectively) to the element $\delta_0:=\frac{-b + \sqrt{d}}{2}$ if $K$ is a field, and the element $\delta_0:= \left(\frac{-b + \sqrt{d}}{2}, \frac{-b - \sqrt{d}}{2}\right)$ if $K=F \oplus F$. One can show (using \cite[Sect. 3]{PS1}) that \begin{equation}\label{e:Delta0OK}\OF_K = \{x + y \delta_0: x \in \OF, y \in \OF\}\end{equation} and for $m \ge 1$, $$U_T(m) = \{x + y \Delta_0: x \in \OF^\times, y \in \p^m\}.$$
Using the above description, it is easy to see that \eqref{isoeqinert}, \eqref{isoeqsplit} induce isomorphisms $$U_T(m) \stackrel{\sim}{\longrightarrow} \OF^\times\left(1+\p^m\OF_K\right)$$ for each $m\ge 1$.

Any character $\Lambda$ of $K^\times$ satisfying $\Lambda|_{F^\times}=1$ can be identified with a character of $F^\times \bs T(F)$. From the above description, it follows that such a character must be trivial on $U_T(m) \simeq \OF^\times\left(1+\p^m\OF_K\right)$ for some $m$. We define $$c(\Lambda) = \min\{m \ge 0: \Lambda|_{U_T(m)} = 1\}.$$

\subsection{Main results}For the rest of Section \ref{s:local}, let $\pi$ be a tempered, spherical, irreducible principal series representation of $\GSp_4(F)$, i.e., $\pi$ is a tempered representation of Type I in the notation of \cite{NF}. We also assume throughout that $\pi$ is of trivial central character. In particular, $\pi$ is the unramified constituent of a representation $\chi_1 \times \chi_2\rtimes \sigma$ induced from a character of the Borel subgroup associated to unramified characters $\chi_1, \chi_2, \sigma$ of $F^\times$ satisfying $\chi_1 \chi_2 \sigma^2 =1$. We put
$$\alpha=\sigma(\varpi), \quad \beta = \sigma(\varpi)\chi_1(\varpi).$$
Note that the temperedness of $\pi$ implies that $\chi_i$ and $\sigma$ are unitary and therefore $|\alpha| = |\beta| = 1.$ Let $\phi$ be a (unique up to multiples) spherical vector in $V_\pi$, i.e., $\phi$ is fixed by the subgroup $\GSp_4(\OF)$.

It is known (see, e.g., Table 2 of \cite{PS1}) that $\mathcal{B}_{\Lambda,\theta}(\pi)$ exists for all characters $\Lambda$ of $F^\times \bs T(F)$. For any such character $\Lambda$ we let $B_{\phi, \Lambda} \in \mathcal{B}_{\Lambda,\theta}(\pi)$ be the element corresponding to $\phi$ under some choice of isomorphism $\pi \simeq \mathcal{B}_{\Lambda,\theta}(\pi)$.

 For $\ell, m \in \Z$, let
\begin{equation}\label{hlmdefeq}
 h(\ell,m)=\begin{bmatrix}\varpi^{\ell+2m}\\&\varpi^{\ell+m}\\&&1\\&&&\varpi^m\end{bmatrix}.
\end{equation}
Using the Iwasawa decomposition, one can show that
\begin{equation}\label{SuganoHFdecompositioneq}
 \GSp_4(F)=\bigsqcup_{\substack{\ell,m\in\Z\\m\geq0}}R(F)h(\ell,m)\GSp_4(\OF);
\end{equation}
cf.\ (3.4.2) of \cite{Fu}.
\subsubsection{Growth of the local Bessel function}
We keep the setup as above. We refer to the function $B_{\phi,\Lambda}$ as the local (spherical) Bessel function of type $(\Lambda, \theta)$. Note that this function depends on our choice of realization of $\pi$ in $\mathcal{B}_{\Lambda,\theta}(\pi)$. However, given $\pi$, $\Lambda$ and $\theta$, this function is canonically specified up to multiples. Our next result bounds the growth of this function along the elements $h(\ell, m)$.
\begin{proposition}\label{p:localboundbesselsugano}Let $\phi \in V_\pi$ be a spherical vector. For a character $\Lambda$ of $F^\times\bs T(F)$, we have $B_{\phi,\Lambda}(h(0, c(\Lambda))) \neq 0$. Furthermore, for non-negative integers $\ell$, $m$  satisfying $m \ge c(\Lambda)$, we have
$$\frac{B_{\phi,\Lambda}(h(\ell, m))}{B_{\phi,\Lambda}(h(0, c(\Lambda)))} \ll (m-c(\Lambda)+1)^3(\ell+1)^3 q^{-2(m - c(\Lambda)) - 3 \ell/2}.$$
The implied constant is absolute.
\end{proposition}

\subsubsection{Growth of the local Bessel integral}
We keep the setup as above. For brevity, denote $$\phi^{(\ell,m)} = \pi(h(\ell,m))\phi.$$ We are interested in bounding the quantity $J_{\Lambda, \theta}(\phi^{(\ell,m)})$ for $m \ge c(\Lambda)$. Due to Proposition \ref{p:localboundbesselsugano} and the discussion in Section \ref{s:besselmodels} it suffices to consider the case $\ell=0$, $m=c(\Lambda)$. We prove the following bound.

\begin{proposition}\label{p:localboundJp}Let $\phi \in V_\pi$ be a spherical vector and $\Lambda$ be a character of $F^\times\bs T(F)$. If the residue field characteristic $q$ is even and $c(\Lambda)=0$, assume that $F= \Q_2$.
We have
$$J_{\Lambda, \theta}(\phi^{(0,c(\Lambda))}) \ll (c(\Lambda) + 1)^{6} q^{-4c(\Lambda)}.$$ The implied constant is absolute.
\end{proposition}

\subsection{Proof of Proposition \ref{p:localboundbesselsugano}}
Fix a character $\Lambda$ of $F^{\times}\backslash T(F)$  and recall the definitions of the parameters $\a, \b$ attached to the representation $\pi$, and of the local spherical Bessel function $B_{\phi,\Lambda} \in \mathcal{B}_{\Lambda,\theta}(\pi)$ of type $(\Lambda,\theta)$. It is known that $B_{\phi,\Lambda}(h(\ell,m))=0$ whenever $\ell <0$ or $m<c(\Lambda)$, cf. \cite[Lemma 3.4.4]{Fu} and \cite[Lemma 2.5]{Sugano1985}.
\subsubsection{A formula due to Sugano} In \cite[Theorem 2.1 and Proposition 2.5]{Sugano1985} Sugano proved that $B_{\phi,\Lambda}(h(0,c(\Lambda)))\neq 0$ and provided a formula for the generating function of $B_{\phi,\Lambda}(h(\ell,m))$. We write it down in a form which will be convenient for the proof of Proposition \ref{p:localboundbesselsugano}. Namely:
\begin{equation}\label{eq:sug-formula}
\sum_{\ell, m \ge 0}B_{\phi,\Lambda}(h(\ell,m+c(\Lambda)))x^my^\ell = B_{\phi,\Lambda}(h(0,c(\Lambda)))\frac{H(x,y)}{P(x)Q(y)} ,
\end{equation}
where
$$P(x)=(1-\a\b q^{-2}x)(1-\a\b^{-1} q^{-2}x)(1-\a^{-1}\b q^{-2}x)(1-\a^{-1}\b^{-1} q^{-2}x)$$
$$Q(y)=(1-\a q^{-3/2}y)(1-\b^{-1} q^{-3/2}y)(1-\a^{-1} q^{-3/2}y)(1-\b^{-1} q^{-3/2}y)$$
and
\begin{itemize}
\item in case $c(\Lambda)>0$:
$$H(x,y)= 1+xq^{-2}-xyq^{-7/2}\sigma(\a,\b) + xy^2q^{-5} + x^2y^2q^{-7}$$
\item in case $c(\Lambda)=0$:
\begin{align*}
H(x,y) = & 1+xq^{-2}(1+\delta(\a,\b)) + x^2q^{-4}\( q^{-1}\(\frac{K}{\p}\) +\delta(\a,\b) +q^{-1/2}\epsilon\sigma(\a,\b)\)\\
& + x^3q^{-7}\(\frac{K}{\p}\)\\
& +y\left[ -q^{-2}\epsilon +xq^{-7/2}\(q^{-1/2}\epsilon\tau(\a,\b)-q^{-1/2}\epsilon -\sigma(\a,\b) \)\right.\\
& \hspace{1cm} + x^2q^{-6}\(\epsilon(\tau(\a,\b) -1-\sigma(\a,\b)^2) -q^{1/2}\delta(\a,\b)\sigma(\a,\b)\)\\
& \hspace{1cm}\left. -x^3q^{-8}\( q^{-1/2}\(\frac{K}{\p}\)\sigma(\a,\b)+\epsilon\)\right]\\
& +y^2\left[-q^{-4}\(\frac{K}{\p}\) + xq^{-5}\( 1+q^{-1}\(\frac{K}{\p}\) (\tau(\a,\b) -2)\)\right.\\
& \hspace{1cm} + x^2q^{-7}\( 1 +\delta(\a,\b) + q^{-1}\(\frac{K}{\p}\)(\sigma(\a,\b)^2-2\tau(\a,\b) +2)\)\\
& \hspace{1cm} \left. +x^3q^{-9}\( q^{-1}\(\frac{K}{\p}\) -q^{-1}\( \frac{K}{\p}\) \delta(\a,\b) +q^{-1}\epsilon^2\)\right]
\end{align*}
\end{itemize}
where
$$\delta(\a,\b)=\frac{q^{1/2}}{q+\(\frac{K}{\p}\)}\( q^{-1/2}\epsilon^2 +2q^{-1/2}\(\frac{K}{\p}\)-\epsilon \sigma(\a,\b) -q^{-1/2}\(\frac{K}{\p}\)\tau(\a,\b) \) ,$$
$$\epsilon =\begin{cases} 0, &\( \frac{K}{\p}\) =-1,\\
\Lambda(\varpi_K), &\( \frac{K}{\p}\) =0,\\
\Lambda((\varpi,1))+\Lambda ((1,\varpi)), &\( \frac{K}{\p}\) =1,\end{cases}$$
with $\varpi_K$ a uniformizer of $\p_K$, and
$$\sigma(\a,\b)=\a +\b +\a^{-1} +\b^{-1} ,$$
$$\tau(\a,\b)=\a\b + \a\b^{-1} + \a^{-1}\b + \a^{-1}\b^{-1}+2 .$$

\subsubsection{The required bound}
The proof of Proposition \ref{p:localboundbesselsugano} will follow from the formula \eqref{eq:sug-formula}. For the sake of brevity, denote:
$$a(m,\ell):=B_{\phi,\Lambda}(h(\ell,m+c(\Lambda)))(B_{\phi,\Lambda}(h(0,c(\Lambda))))^{-1},\qquad
H(x,y)=\sum_{m=0}^3\sum_{\ell=0}^2 h_{m,\ell}x^my^\ell .$$
Recall that  $|\a|=|\b|=1$. Hence, since $|\epsilon|\leq 2$, it follows that $|\delta(\a,\b)|< 28 q^{-1/2}$
and we may derive good bounds for $|h_{m,\ell}|$; they are listed in Table \ref{table:bound-for-h(m,l)}.
\begin{table}[ht] 
\caption{The values of $\Xi_{m,\ell}$ for which $|h_{m,\ell}|\leq\Xi_{m,\ell}$.}\label{table:bound-for-h(m,l)}
\setlength{\extrarowheight}{0.1cm}
\begin{subtable}{.9\linewidth}
\centering
\begin{tabular}{|c|c|c|c|}
\hline
\diagbox[height=0.7cm]{$m$}{$\ell$} & 0 & 1 & 2 \\
\hline
0					& $1$ & $2q^{-2}$ & $q^{-4}$\\
\hline
1 				& $q^{-2}+28q^{-5/2}$ & $10q^{-4}+4q^{-7/2}$ & $q^{-5}+4q^{-6}$ \\
\hline
2 				& $q^{-5}+36q^{-9/2}$ & $134q^{-6}$ & $q^{-7}+28q^{-15/2}+6q^{-8}$ \\
\hline
3 				& $q^{-7}$ & $2q^{-8}+4q^{-17/2}$ & $5q^{-10}+28q^{-21/2}$ \\\hline
\end{tabular}
\caption{In case $c(\Lambda)=0$.}
\end{subtable}\\\vspace{0.5cm}
\begin{subtable}{.9\linewidth}
\centering
\begin{tabular}{|c|c|c|c|}
\hline
\diagbox[height=0.7cm]{$m$}{$\ell$} & 0 & 1 & 2 \\
\hline
0					& $1$ & $0$ & $0$\\
\hline
1 				& $q^{-2}$ & $4q^{-7/2}$ & $q^{-5}$ \\
\hline
2 				& $0$ & $0$ & $q^{-7}$ \\
\hline
3 				& $0$ & $0$ & $0$ \\\hline
\end{tabular}
\caption{In case $c(\Lambda)>0$.}
\end{subtable}%
\end{table}

\noindent Note in particular that for all $m,\ell\geq 0$:
$$h_{m,\ell}\ll q^{-2m-\frac32 \ell} .$$
Using the geometric series expansion of $(P(x)Q(y))^{-1}$ in the formula \eqref{eq:sug-formula}, we obtain
\begin{align*}
a(m,\ell) = & \sum_{\tilde{m}=0}^3\sum_{\tilde{l}=0}^2\(\sum_{\substack{(m_1,m_2,m_3,m_4)\in\Z^4_{\geq 0}\\ m_1+m_2+m_3+m_4=m-\tilde{m}}}\hspace{-1cm} \alpha^{m_1+m_2-m_3-m_4}\beta^{m_1-m_2+m_3-m_4}\)\(\sum_{\substack{(l_1,l_2,l_3,l_4)\in\Z^4_{\geq 0}\\ l_1+l_2+l_3+l_4=\ell-\tilde{l}}}\hspace{-0.5cm} \alpha^{l_1-l_3}\beta^{l_2-l_4}\)\\
&\hspace{0.5cm}\times h_{\tilde{m},\tilde{l}}q^{-2(m-\tilde{m})}q^{-\frac32 (\ell-\tilde{l})}
\end{align*}
Hence, because $|\a|=|\b|=1$,
\begin{equation}\label{eq:explicit-bound}
|a(m,\ell)q^{2m+\frac32 \ell}|\leq \sum_{\tilde{m}=0}^3\sum_{\tilde{l}=0}^2 (m-\tilde{m}+1)^3(\ell-\tilde{l}+1)^3 |h_{\tilde{m},\tilde{l}}|q^{2\tilde{m}+\frac32 \tilde{l}} ,
\end{equation}
and from Table \ref{table:bound-for-h(m,l)} we see that $|h_{\tilde{m},\tilde{l}}|q^{2\tilde{m}+\frac32 \tilde{l}}$ is bounded above by an absolute and effective constant. This proves that
$$B_{\phi,\Lambda}(h(\ell,m+c(\Lambda)))\ll (m+1)^3(\ell+1)^3 q^{-2m-\frac32 \ell}|B_{\phi,\Lambda}(h(0,c(\Lambda)))|$$
where the implied constant is absolute and effective.

\subsection{Bounding the local Bessel integral}\label{s:proofprop2.2}

\subsubsection{A set of coset representatives}
As a starting point, we write down for each $m\ge 1$ a set of coset representatives for  $F^\times U_T(m) \bs T(F) \simeq  F^\times (1+\p^m \OF_K)  \bs  K^\times$. Recall the definition of $\Delta_0$ from Section \ref{s:levelstructure}.

In the next lemma when we write $y \in \OF/\p^m$ or $y \in \p/\p^m$, we mean that $y$ ranges over a set of coset representatives in $\OF$ for these quotients. Furthermore, we choose the coset representative for 0 to have valuation $m$ (this is not required, but simplifies the exposition of the proof).
\begin{lemma}\label{representatives}
For each integer $m>0$, let the sets $D$ and $S_m$ be as defined below.
\begin{enumerate}
    \item if $K/F$ is an unramified field extension, let  $D = \{1\}$ and
 $$S_m=\{1+y\Delta_0 : y \in \OF/\p^m \}  \cup \{x+\Delta_0 : x \in \p/\p^m\}.$$
     \item if $K/F$ is a ramified field extension, choose  $P_0 = x_0 + y_0\Delta_0 = \mat{x_0}{y_0c}{-y_0a}{x_0 - y_0b}$ such that $x_0, y_0 \in \OF$ and $\det(P_0) \in \varpi\OF^\times$ (such a choice is possible by \eqref{e:Delta0OK}). Let $ D=\{1,P_0\}$ and
    $$S_m=\{1+y\Delta_0 : y \in \OF/\p^m \}.$$
    \item if $K=F \times F$, recall that $\sqrt{d} \in \OF^\times$, $\frac{b+\sqrt{d}}{2} \in \OF$, and  let  $P_0 = 1+ \frac{b+\sqrt{d}}{2\sqrt{d}}(\varpi -1) + \frac{\varpi-1}{\sqrt{d}}\Delta_0$. Let $ D=\{P_0^n: n \in \Z\}$ and
    $$S_m=\{1+\frac{b+\sqrt{d}}{2}y+ y\Delta_0 : y \in \OF/\p^m, y \notin \frac{-1}{\sqrt{d}} +\p \}.$$
  \end{enumerate}
Then in each case the set $D S_m := \{s_1s_2: s_1 \in D, \ s_2 \in S_m \}$ gives a complete set of representatives for the quotient $F^\times U_T(m) \bs T(F)$.
\end{lemma}
\begin{proof}It suffices to show that the image of $D$ under \eqref{isoeqinert} or \eqref{isoeqsplit} gives a complete set of representatives for $F^\times \OF_K^\times\bs K^\times$ and the image of $S_m$ under \eqref{isoeqinert} or \eqref{isoeqsplit} gives a complete set of representatives for $\OF^\times(1+\p^m \OF_K)\bs \OF_K^\times$.

{\bf \underline{Inert case:}} Assume $K/F$ is an unramified extension.
Then $F^\times \OF_K^\times = K^\times$, which agrees with $D =\{1\}$. To show that the image of $S_m$ under \eqref{isoeqinert} gives a complete set of representatives for $\OF^\times(1+\p^m \OF_K)\bs \OF_K^\times$, note first that $S_m$ has the correct cardinality by Lemma 3.5.3 of \cite{Fu} and each element of $S_m$ maps to $\OF_K^\times$. Using \eqref{e:Delta0OK}, it is easy to see any element of $\OF_K^\times$ can upon multiplying by a suitable element of $\OF^\times(1+\p^m \OF_K)$ be brought to one of the elements in the image of $S_m$. The proof is complete.

{\bf \underline{Ramified case:}} Assume $K/F$ is a ramified extension. Note that the image of $P_0$ under \eqref{isoeqinert} is an uniformizer $\varpi_K$ of $\OF_K$. So $D$ maps onto $\{1, \varpi_K\} \simeq F^\times \OF_K^\times\bs K^\times$. The proof that the image of $S_m$ under \eqref{isoeqinert} gives a complete set of representatives for $\OF^\times(1+\p^m \OF_K)\bs \OF_K^\times$ is essentially identical to the inert case.

{\bf \underline{Split case:}} Assume $K=F \times F$. A computation shows that the image of $P_0$ under \eqref{isoeqsplit} is the element  $(\varpi, 1)$. Therefore $D$ maps onto the set $\{(\varpi^n, 1): n\in \Z\}$ which is clearly a complete set of representatives for  $F^\times (\OF^\times \times \OF^\times) \bs (F^\times \times F^\times).$ Similarly, the image of $S_m$ under \eqref{isoeqsplit} is the
set $\{(1 + \sqrt{d}y, 1): y \in \OF/\p^m, y \notin \frac{-1}{\sqrt{d}} +\p \}$ which clearly
 gives a complete set of representatives for $\OF^\times \left((1+\p^m \OF) \times (1+\p^m \OF) \right)\bs (\OF^\times \times \OF^\times)$.
\end{proof}
\subsubsection{Some volume computations} We now calculate some volumes
that will appear in our calculation of the local integral.
\begin{lemma}\label{TrivialDeterminant}
    Let $X,S$ be two $2 \times 2$ matrices.
    Then
    $$\det(X+S)=\det(X)+\trace(^t{A_S}X)+\det(S),$$
    where $A_S$ is the adjugate matrix of $S$.
    In particular, if $S$ is invertible, then
    $$\det(X+S)=\det(X)+(\det S)\trace(S^{-1}X)+\det(S).$$
\end{lemma}
\begin{proof}
    Say $X=\mat{x}{y}{w}{z}$ and $S=\mat{a}{b}{c}{d}$.
    Then
    \begin{align*}
        \det(X+S)&=(x+a)(z+d)-(y+b)(w+c)\\
        &= (ad-bc)+(az+xd-yc-wb)+(xz-yw).
    \end{align*}
    Now
    \begin{align*}
        az+xd-yc-wb&=\trace\mat{xd-wb}{*}{ *}{az-yc}\\
        &= \trace\left(\mat{d}{-b}{-c}{a}\mat{x}{y}{w}{z}\right),
    \end{align*}
    which proves the first claim.
    The second claim comes from the identity
    $$^t{A_S}=(\det S)S^{-1}.$$
\end{proof}
\begin{lemma}\label{Volumes}
Let $a,b,c,d$ be integers.
Let $$\Sc_{a,b,c,d}=\{\mat{x}{y}{y}{z} \in M^{\text{sym}}_2(F): v(x) \ge a, v(y) \ge b, v(z) \ge c, v(xz-y^2) \ge d\}.$$
Then we have
$$\vol(\Sc_{a,b,c,d}) \le q^{-\max\{d/2,b\}-\max\{d,a+c\}}+2\max\{0,d-a-c\}q^{-d-b}.$$
\end{lemma}
\begin{proof}
We have $\Sc_{a,b,c,d} \subset S_1 \cup S_2$ where
$$S_1 = \{\mat{x}{y}{y}{z} \in M^{\text{sym}}_2(F): v(x) \ge a, v(y) \ge b, v(z) \ge c, \min\{v(xz), 2v(y)\} \ge d\}$$
and
$$S_2=\{\mat{x}{y}{y}{z} \in M^{\text{sym}}_2(F): v(x) \ge a, v(y) \ge b, v(z) \ge c, v(xz) < d, xz \in y^2 + \p^d\}.$$

    {\bf \underline{Volume of $S_1$:}} Let $j=\max\{b,\lceil \frac{d}{2} \rceil\}$, so
that if $\mat{x}{y}{y}{z} \in S_1$ then $y \in \p^j$.
If $a+c \ge d$ then we have
    $$\vol(S_1)= \vol(\p^a \times \p^j \times \p^c) \le q^{-\max\{d/2,b\}-(a+c)}.$$
On the other hand, if $a+c < d$ then let $i=v(x)$.
Since $v(xz) \ge d$, we must have $v(z) \ge d-i$, which becomes automatic
from the condition $v(z) \ge c$ as soon as $i \ge d-c$. Thus we have
\begin{align*}
    \vol(S_1)&=\vol(\p^j)\left(\sum_{i=a}^{d-c-1} \vol(\varpi^i \OF^\times \times \p^{d-i}) + \vol(\p^{d-c} \times \p^c) \right)\\
    &=q^{-j} ((d-a-c)(1-q^{-1})q^{-d} + q^{-d}) \le (1+d-a-c)q^{-\max\{d/2,b\}-d}.
\end{align*}
Combining the cases $a+c \ge d$ and $a+c<d$ together, we find that
$$ \vol(S_1) \le (1+\max\{0,d-a-c\})q^{-\max\{d/2,b\}-\max\{d,a+c\}}.$$

{\bf \underline{Volume of $S_2$:}}
Assume $\mat{x}{y}{y}{z} \in S_2$ and let $i=v(x)$ and $j=v(y)$.
Since $v(xz)<d$, the condition $xz \in y^2 +\p^d$ forces $2j=i+v(z)<d$
and in particular $2j-i \ge c$.
Moreover, for fixed $x$, we must have $z \in \frac{y^2}{x} + \p^{d-i} $.
Therefore,
\begin{align*}
    \vol(S_2)&= \sum_{j \ge b} \vol(\varpi^j \OF^\times) \sum_{i \ge a}
    \vol(\varpi^i \OF^\times \times \p^{d-i}) \1_{\substack{2j-i \ge c\\ 2j <d}}\\
    &= \sum_{j=b}^{\lfloor\frac{d-1}2\rfloor}(1-q^{-1})q^{-j}\sum_{i=a}^{2j-c}(1-q^{-1})q^{-d}\\
    & \le \max\{0,d-c-a\}q^{-b-d}.
\end{align*}
Therefore, in total we have
\begin{align*}
    \vol(\Sc_{a,b,c,d}) &\le q^{-\max\{d/2,b\}-\max\{d,a+c\}} + \max\{0,d-c-a\}(q^{-\max\{d/2,b\}-d}+q^{-b-d})\\
    &\le  q^{-\max\{d/2,b\}-\max\{d,a+c\}} + 2\max\{0,d-c-a\}q^{-b-d}.
\end{align*}
\end{proof}

\begin{definition}\label{CartanCoordinates}
Fix $j$ a non-negative integer.
For $Y=\mat{x}{y}{y}{z} \in M_2^{\text{sym}}(F)$, define
    $$M_1(Y;j)=\min\{0,v(x)+j,v(y),v(z)\}$$ and
    $$M_2(Y;j)=\min\{0,j+v(xz-y^2), j+v(x),j+v(y),v(z)\}.$$
\end{definition}
For future use, we note the following inequality, holding for all
$Y \in M_2^{\text{sym}}(F)$
\begin{equation}\label{InequalityMM}
2M_1(Y;j) \le M_2(Y,j) \le M_1(Y;j)+j.
\end{equation}
The second inequality comes from
$$M_2(Y,j) \le \min\{0,j+v(x),j+v(y),v(z)\}
\le  \min\{0,j+v(x),v(y),v(z)\}+j.$$
To establish the first inequality, observe that
\begin{align*}
    v(xz-y^2) &\ge \min\{v(x)+v(z),2v(y)\}
\end{align*}
and thus
$$M_2(Y,j) \ge \min\{0, v(x)+v(z)+j,2v(y),j+v(x),j+v(y),v(z)\} \ge 2M_1(Y;j).$$

\begin{lemma}\label{PainfulVolumes}
For any integers $m_1,m_2 \le 0$ with $2m_1\le m_2$ define
$$N(m_1,m_2;j)=\{Y \in M_2^{\text{sym}}(F): M_1(Y;j)=m_1, M_2(Y;j)=m_2\}.$$
Then we have
$$\vol\left( N(m_1,m_2;j)\right) \le (1+2\min\{m_2-2m_1,-m_2\})q^{-m_1-m_2+j}.$$
\end{lemma}
\begin{proof}
    If $Y \in N(m_1,m_2;j) $ then by definition of $M_1(Y;j)$ we
    must have $v(x) \ge m_1-j$, and by definition of $M_2(Y;j)$
    we must also have $v(x) \ge m_2-j$.
    Thus $v(x) \ge \max \{m_1,m_2\}-j$.
    By the same reasoning, it follows that
    $$Y \in \Sc_{a,b,c,d},$$
    where
    $$a=\max\{m_1,m_2\}-j, b=\max\{m_1,m_2-j\}, c=\max\{m_1,m_2\}, d=m_2-j.$$
    We have
    $$\max\{d/2,b\}=\max\left\{m_1,m_2-j,\frac{m_2-j}2\right\}
    =\max\left\{m_1,\frac{m_2-j}2\right\} \ge m_1.$$
    Moreover, using the fact that $m_2 \le 0$, we have by~(\ref{InequalityMM})
    \begin{align*}
        \max\{d,a+c\}&=\max\{m_2-j,2\max\{m_1,m_2\}-j\}\\
    &=\max\left\{m_2,2m_1\right\}-j=m_2-j.
    \end{align*}
    Thus, by Lemma~\ref{Volumes},
    we obtain
    $$\vol\left( N(m_1,m_2;j)\right) \le (1+2\min\{m_2-2m_1,-m_2\})q^{-m_1-m_2+j}.$$
\end{proof}

\subsubsection{Proof of Proposition \ref{p:localboundJp}} Recall that $\pi$ is an irreducible, tempered, spherical representation,  $\phi \in V_\pi$ is a spherical vector, and $\Lambda$ is a character of $K^\times$ satisfying $\Lambda|_{F^\times}=1$.

Our aim is to prove  Proposition \ref{p:localboundJp}. We assume that $c(\Lambda) \ge 1$ because the case $c(\Lambda) = 0$ is known by Theorem 2.1 of \cite{DPSS15}.
Our bound will follow from the calculation of explicit coset representatives
in Lemma~\ref{representatives} together with a bound for the following stable integral.
Given a positive integer $m$ and given $t \in T(F)$, we define
$$J_0(\phi^{(0,m)},t)= \int\limits_{N(F)}^{\st}\Phi_{\phi^{(0,m)}}(nt)\theta^{-1}(n)\, dn.$$

\begin{proposition}\label{p:GeneralBound}
Let $t \in T(F)$ and let $m$ be a positive integer.
Write the Cartan decomposition
\begin{equation}\label{CartandDecomposition}
   \mat{\varpi^{-2m}}{}{}{\varpi^{-m}} t \mat{\varpi^{2m}}{}{}{\varpi^{m}}=UDV
\end{equation}
with
$U,V \in \GL_2(\OF)$, $D=\mat{\varpi^i}{}{}{\varpi^{i+j}}$ and $j \ge 0$.
Then we have
$$J_0(\phi^{(0,m)},t)\ll (m+v(2)+j)^5 q^{-3m-\frac{j}2}.$$
\end{proposition}

\begin{proof}
Since $\Phi_{\phi^{(0,m)}}$ is bi-$F^\times h(0,m)\GSp_4(\OF)h(0,m)^{-1}$-invariant,
we have
\begin{equation}\label{SeriesJ0}
    J_0(\phi^{(0,m)},t)=\lim_{k\to\infty}
\sum_{\ell',m'} \Phi_{\phi^{(0,m)}} (h(\ell',m'))
\int_{N(\p^{-k}) \cap N(\ell',m';m,t)} \theta^{-1}(n)\, dn,
\end{equation}
where
$$N(\ell',m';m,t)=\{n \in N(F) : h(0,m)^{-1}nt h(0,m)\in Z(F)\GSp_4(\OF)h(\ell',m')\GSp_4(\OF)\}.$$
Now let $n=\mat{1}{X}{}{1} \in N(F)$.
For convenience, define
$$h(m)=\mat{\varpi^{2m}}{}{}{\varpi^m},$$
so that
$$h(0,m)=\mat{h(m)}{}{}{\varpi^{2m}h(m)^{-1}}.$$
Then
\begin{align*}
    &h(0,m)^{-1}nth(0,m)\\
    &=
    h(0,m)^{-1}
    \mat{1}{X}{}{1}
    \mat{t}{}{}{(\det t)^t{t}^{-1}}
    h(0,m)\\
    &=
    \mat{h(m)^{-1}}{}{}{\varpi^{-2m}h(m)}
    \mat{1}{X}{}{1}
    \mat{h(m)UDVh(m)^{-1}}{}{}{(\det t){h(m)^{-1}} ^t{U}^{-1}{D^{-1}}^t{V}^{-1}h(m)}
    \mat{h(m)}{}{}{\varpi^{2m}h(m)^{-1}}\\
    & =\mat{U}{}{}{^tU^{-1}}\mat{D}{(\det t) Y D^{-1}}{}{(\det t) D^{-1}}\mat{V}{}{}{^t{V}^{-1}},
\end{align*}
where we have set
$U^{-1} \mat{\varpi^{-2m}}{}{}{\varpi^{-m}} X \mat{1}{}{}{\varpi^{m}} ^tU^{-1}=Y=\mat{\tilde{x}}{\tilde{y}}{\tilde{y}}{\tilde{z}}$.
Note that there is a unit $u$ such that $\det t= \varpi^{2i+j} u$, and
\begin{align*}
    \mat{D}{(\det t) Y D^{-1}}{}{(\det t) D^{-1}}
    &=\mat{1}{Y}{}{1}\mat{D}{}{}{(\det t) D^{-1}}\\
&=\varpi^i \mat{1}{Y}{}{1}
\begin{bmatrix}1\\&\varpi^{j}\\&&\varpi^j\\&&&1\end{bmatrix}
\begin{bmatrix}1\\&1\\&&u\\&&&u\end{bmatrix}.
\end{align*}
Thus, recalling Definition~\ref{CartanCoordinates}, by~\cite[Lemma 2.7]{DPSS15} we have $n \in N(\ell',m';m,t)$ if and only if
$$\ell'=j+2M_1(Y;j)-2M_2(Y;j)$$
and
$$m'=M_2(Y;j)-2M_1(Y;j).$$
Henceforth, set $m_1=-m'-\frac{\ell'}2+\frac{j}2$ and
$m_2=j-m'-\ell'$, so that $n \in N(\ell',m';m,t)$ if and only if $Y \in N(m_1,m_2;j)$.
Observe that
\begin{align*}
    \trace(SX)&=\trace(S\mat{\varpi^{2m}}{}{}{\varpi^m}UY^t{U}\mat{1}{}{}{\varpi^{-m}})\\
    &=\trace(S'Y),
\end{align*}
where
$$S'=^t{U}\mat{1}{}{}{\varpi^{-m}}S\mat{\varpi^{2m}}{}{}{\varpi^m}U=
^t{U}\mat{\varpi^m}{}{}{1}S\mat{\varpi^{m}}{}{}{1}U.$$
Thus, changing variables $X \mapsto Y$, we have
\begin{equation}\label{SliceTheIntegral}
     \int_{N(\p^{-k}) \cap N(\ell',m';m,t)} \theta_{S}^{-1}(n)\, dn=q^{-3m}I^{(k)}(m_1,m_2;j),
\end{equation}
where
$$I^{(k)}(m_1,m_2;j)=\int_{N(m_1,m_2;j) \cap S_U(k,m)}
\psi(-\trace(S'Y )) \, dY$$
and
$$S_U(k,m)= U^{-1} \mat{\varpi^{-2m}}{}{}{\varpi^{-m}} M_2^{\text{sym}}(\mathfrak p^{-k}) \mat{1}{}{}{\varpi^{m}} ^t{U}^{-1}.$$
Since
$$ \left| I^{(k)}(m_1,m_2;j) \right| \le \vol(N(m_1,m_2;j)),$$
by Lemma~\ref{PainfulVolumes} we obtain
\begin{equation}\label{VolumeBound}
   \left| \int_{N(\p^{-k}) \cap N(\ell',m';m,t)} \theta_{S}^{-1}(n)\, dn \right|  \le
    (1+2\min\{m',\ell'+m'-j\})q^{\frac32 \ell'+2m'-\frac{j}2-3m}.
\end{equation}
By Macdonald's formula~\cite[Proposition 2.10]{DPSS15} we have for all non-negative integers $\ell',m'$
\begin{equation}\label{McDonaldBound}
   \Phi_{\phi^{(0,m)}}(h(\ell',m')) =  \Phi_{\phi}(h(\ell',m')) \ll (m'+\ell')^2 q^{-(4m'+3\ell')/2}.
\end{equation}
The bound~(\ref{McDonaldBound}) follows from the fact that the
Macdonald's formula is essentially a double geometric sum in the Satake parameters, of length $\ll m'+\ell'$.
Hence combining~(\ref{McDonaldBound}) and~(\ref{VolumeBound}),
every term in~(\ref{SeriesJ0}) is $\ll (m'+\ell')^3 q^{-\frac{j}2-3m}$.
It remains to show that the series~(\ref{SeriesJ0}) is actually a finite sum.

From now on, assume that $m_1+j \le -2m-6-j-2v(2)$.
We claim that for $k$ large enough we have
$I^{(k)}(m_1,m_2;j)=0$. This is enough to prove the Proposition, because then only the terms with $0 \ge m_1>-2m-6-2j-2v(2)$
will contribute to~(\ref{SeriesJ0}), but in view of inequality~(\ref{InequalityMM}) these terms
must also satisfy $0 \ge m_2 \ge -4m-12-4j-4v(2)$, and thus there are $\ll (m+v(2)+j)^2$ such terms.

We split the integral $I^{(k)}(m_1,m_2;j)$ over the following two disjoint ranges:
\begin{enumerate}
    \item $v(\trace(S'Y)) <-2$,
    \item $v(\trace(S'Y)) \ge -2$.
\end{enumerate}
For $i \in \{1,2\}$, let $I_i$ be the integral in the corresponding range.
Range~(1) is trivially stable by the change of variable $Y \mapsto \lambda Y$ for any $\lambda \in \OF^\times$.
Thus
$$I_1 = \int \psi(-\lambda\trace(S'Y)) \, dY.$$
Integrating both sides with respect to $\lambda \in 1+\p$ gives $I_1=0$.

Next, we claim that range~(2) is stable by the change of variables $Y \mapsto Y+\varpi^{-1-v(2)}S'^{-1}$.
To prove this, we have three conditions to check:
\begin{enumerate}
\item We have $$v[\trace(S'(Y+\varpi^{-1-v(2)}S'^{-1}))]=v(\trace(S'Y)+2\varpi^{-1-v(2)}) \ge -2.$$
\item We need to check that $M_1(Y+\varpi^{-1-v(2)}S'^{-1};j)=M_1(Y;j)=m_1$.
We have $\det(S') \in \varpi^{2m} \OF^\times \det(S)=\frac{\varpi^{2m}d}4 \OF^\times$ and by Assumption~\ref{standardassumptions},
$$v(\det(S')) \le 2m+1.$$
Moreover, since the entries of $S$ are integers, it is clear from
the definition that $S'$ also has integer coefficients.
Since $$S'^{-1}={\frac1{\det S'}}{}^t{A_{S'}},$$
where $A_{S'}$ is the adjugate matrix of $S'$, it follows that
all the entries of $\varpi^{-1-v(2)}S'^{-1}$ have valuation larger than $-2m-2-v(2)$.
Since we are assuming $M_1(Y;j)=m_1 \le -2m-6-2j-2v(2)< -2m-2-v(2)-j$, it follows
that $M_1(Y+\varpi^{-1-v(2)}S'^{-1};j)=M_1(Y;j)$, as required.
\item We need to check that $M_2(Y+\varpi^{-1-v(2)}S'^{-1};j)=M_2(Y;j)=m_2$.
By Lemma~\ref{TrivialDeterminant} we have
$$\det(Y+\varpi^{-1-v(2)}S'^{-1})=
\det(Y)+\varpi^{-1-v(2)}\det(S')^{-1}\trace(S'Y)+\varpi^{-2-2v(2)}\det(S')^{-1}.$$
In particular,
$$v(\det(Y+\varpi^{-1-v(2)}S'^{-1}))=v(\det(Y))$$
unless
$v(\det(Y)) \ge -2m-5-2v(2)$,
in which case it does not contribute to $M_2(Y;j)$ anyway since
by~(\ref{InequalityMM}), we have $M_2(Y;j)=m_2 \le m_1+j < -2m-5-j-2v(2)$.
In both cases, we have $M_2(Y+\varpi^{-1-v(2)}S'^{-1};j)=M_2(Y;j)$.
\end{enumerate}
Thus, we obtain $$I_2=\psi(2\varpi^{-1-v(2)})I_2.$$
Since $2\varpi^{-1-v(2)}$ generates $\p^{-1}/\OF$ and since $\psi$ is trivial on $\OF$ but not trivial on $\p^{-1}$, we have $\psi(2\varpi^{-1-v(2)}) \neq 1$. This implies $I_2=0$.
\end{proof}

We now proceed to prove Proposition~\ref{p:localboundJp} in all cases -- split, inert and ramified. Until the end of the section, we set $m=c(\Lambda)$.
We remind the reader that Assumptions~\ref{standardassumptions} are in force throughout.
By $U_T(m)$-invariance, we have
\begin{equation}\label{J0Lambda}
    J_{\Lambda,\theta}(\phi^{(0,m)})=\vol(\OF^\times \bs U_T(m))\sum_i\Lambda^{-1}(t_i)J_0(\phi^{(0,m)},t_i),
\end{equation}
where the $i$-sum ranges over representatives $t_i$ of $F^\times U_T(m) \bs T(F)$.
Note that $\vol(\OF^\times \bs U_T(m)) \asymp q^{-m}$.

 {\bf \underline{Proof in the inert case:}}
Assume $K/F$ is an unramified field extension.
Let $t$ be a representative as in Lemma~\ref{representatives}.
Consider first the case $$t=\mat{1}{yc}{-ya}{1-yb}.$$
Write $$\mat{\varpi^{-2m}}{}{}{\varpi^{-m}} t \mat{\varpi^{2m}}{}{}{\varpi^{m}}
=\mat{1}{\varpi^{-m}yc}{-\varpi^{m}ya}{1-yb}=UDV$$
as in~(\ref{CartandDecomposition}).
Noting that $v(y)-m \le 0$, we have
\begin{align*}
    i&=\min\{v(1), v(y)-m, v(ya)+m, v\left(1-yb\right)\}\\
    &=v(y)-m.
\end{align*}
On the other hand, since $t \in S_m$ and the image of $S_m$ under~(\ref{isoeqinert})
is contained in $\OF_K^\times$, we have
$$2i+j=v(\det(t))=v(\Norm(1+y\delta_0))=0.$$
Therefore $$j=2m-2v(y).$$
Thus by Proposition~\ref{p:GeneralBound} we have $J_0(\phi^{(0,m)},t) \ll (m+v(2))^5q^{-4m+v(y)}$.
Since there are $\asymp q^{m-v(y)}$ representatives of this form, their total contribution (taking into account the volume)
to~(\ref{J0Lambda}) is $\ll (m+v(2))^5 q^{-4m}$. Summing this over $0 \le v(y) \le m$ gives
$\ll m(m+v(2))^5 q^{-4m}.$
Next, consider the case $$t=\mat{x}{c}{-a}{x-b}$$
with $x \in \p$. Then similarly as before, we find $i=-m$ and $2i+j=0$.
Thus by Proposition~\ref{p:GeneralBound} we have $J_0(\phi^{(0,m)},t) \ll (m+v(2))^5q^{-4m}$.
Since there are $\asymp q^{m-1}$ representatives of this form, their total contribution
to~(\ref{J0Lambda}) is $\ll (m+v(2))^5 q^{-4m-1}$.

Adding up the contributions of all the representatives thus gives
$$ J_{\Lambda,\theta}(\phi^{(0,m)}) \ll m(m+v(2))^5 q^{-4m}.$$

{\bf \underline{Proof in the ramified case:}}
Now assume $K/F$ is a ramified field extension.
Let $t$ be a representative as in Lemma~\ref{representatives}.
Consider first the case $$t=\mat{1}{yc}{-ya}{1-yb} \in S_m.$$
Then the exact same argument as in the inert case goes through.
Next, consider the case
$$t=(x_0+y_0\Delta_0)(1+y\Delta_0)=\mat{x_0-yy_0ac}{c[y_0+y(x_0-y_0b)]}{-a[y_0+y(x_0-y_0b)]}{x_0+y[y_0(b^2-ac)-bx_0]-by_0}.$$
Recall that $x_0,y_0$ are any two elements of $\OF$ such that
\begin{equation}\label{x0y0}
    v(x_0(x_0-y_0b)+y_0^2ac)=1.
\end{equation}
We shall now choose some convenient $x_0,y_0$.
Let us now distinguish some cases.
Consider first the case $v(a)=1$.
Then we can take $y_0=1$ and $x_0=y_0b$.
Thus $v([y_0+y(x_0-y_0b)])=0$, and hence $i=-m$. Furthermore
$$2i+j=v(\det(t))=v(\det(P_0))=1,$$
and hence $j=2m+1$. From there on, the argument proceeds as in the second case of the inert case, giving a final contribution $\ll (m+v(2))^5q^{-4m-\frac12}$

Next suppose $v(a)>1$.
First, we claim that $v(b)=0$.
Indeed, if $b=\varpi b'$ with $b' \in \OF$
then we have $d=\varpi^2d'$ where $d'=b'^2-4\frac{a}{\varpi^2}c\in\OF$ and thus
$d$ does not generate the discriminant of $K/F$, contradicting Assumption~\ref{standardassumptions}. Thus we can take $y_0=1$ and $x_0=\varpi+b$.
Then $v(y_0+y(x_0-y_0b))=0$ again and the rest of the argument is the same as the  case $v(a)=1$.

Finally consider the case $v(a)=0$. For~(\ref{x0y0}) to hold,
we must have $v(y_0)=v(x_0-y_0b)=0$.
For each $0 \le k \le m$, and given a fixed set $S$ of coset representatives for $\OF/\p^m$ chosen according to our conventions, the number of elements  $y\in S$ satisfying $v(y_0+y(x_0-y_0b))=k$ equals  \[\begin{cases} q^{m-k}(1-q^{-1}) &\text{if $k<m$},\\ 1 &\text{if $k=m$}.\end{cases}\]
(To see the above, put $s=(x_0 - y_0b)$ and suppose that $v(y_0+ys)=k$. If $k=m$, note that  $y$ must be the representative of $-y_0s^{-1}$ in $S$. If $k <m$, write $y = s^{-1}(\varpi^k u - y_0)$ and observe that $u \in \OF^\times$ can take exactly one value in each coset of $\OF^\times/(1+\p^{m-k})$.)
From there on, the argument proceeds as in the first case of the inert case,
 giving a final contribution $\ll m(m+v(2))^5q^{-4m-\frac12}$.

Adding up the contributions of all the representatives again gives
$$ J_{\Lambda,\theta}(\phi^{(0,m)}) \ll m(m+v(2))^5 q^{-4m}.$$

{\bf \underline{Proof in the split case:}}
Finally assume $K=F \times F$.
We start with a technical lemma.
\begin{lemma}\label{IfqEven}
   Assume $v(2)>0$. Then we have $v(\sqrt{d}-b),v(\sqrt{d}+b) \ge v(2)$.
\end{lemma}
\begin{proof}
    We have $d \in b^2+4\OF$ thus
\begin{equation} \label{I1}
    v(\sqrt{d}+b)+v(\sqrt{d}-b) \ge v(4).
\end{equation}
But
\begin{equation}\label{I2}
    v(\sqrt{d}+b)=v(\sqrt{d}-b+2b) \ge \min\{v(2),v(\sqrt{d}-b)\}.
\end{equation}
If $v(\sqrt{d}-b) \le v(2)$ then equation~(\ref{I1}) implies
$v(\sqrt{d}+b) \ge v(2)$, while if $v(\sqrt{d}-b) \ge v(2)$ then equation~(\ref{I2}) implies $v(\sqrt{d}+b) \ge v(2)$.
The same reasoning gives $v(\sqrt{d}-b) \ge v(2)$.
\end{proof}

Now let $t$ be a representative as in Lemma~\ref{representatives},
and write
$t=x+y\Delta=\mat{x+y\frac{b}2}{yc}{-ya}{x-y\frac{b}2}$
for some $x,y \in F$.
Since by~(\ref{isoeqsplit}) the image of $P_0$ is $(\varpi^n,1)$
and the image of $S_m$ is $(\OF/\p^m)^\times \times\{1\}$,
we have
$$\begin{cases}
    x+y\frac{\sqrt{d}}2=\varpi^ku,\\
     x-y\frac{\sqrt{d}}2=1
\end{cases}$$
for some $k \in \Z$ and $u \in (\OF/\p^m)^\times$.
Equivalently,
$$\begin{cases}
    x=\frac{\varpi^ku+1}2,\\
    y=\frac{\varpi^ku-1}{\sqrt{d}}.
\end{cases}$$
We start with the case $k<0$.
Then $v(y)=k$.
Furthermore, we claim that $v(x \pm y \frac{b}2) \ge k$.
If $v(2)=0$ this is obvious.
On the other hand, we have
\begin{align*}
    x+y\frac{b}2&=\frac{\varpi^ku+1}{2}+b\frac{\varpi^ku-1}{2\sqrt{d}}\\
    &=\frac{(\sqrt{d}+b)\varpi^ku+(\sqrt{d}-b)}{2\sqrt{d}}.
\end{align*}
If $v(2)>0$, Lemma~\ref{IfqEven} proves the claim, and similarly for
$x-y\frac{b}2$.
It follows that
\begin{align*}
    i&=\min\left\{v\left(x+y\frac{b}2\right), v(y)-m, v(ya)+m, v\left(x-y\frac{b}2\right)\right\}\\
    &=k-m.
\end{align*}
Moreover, we have $$2i+j=v\left(x^2-\frac{d}{4}y^2\right)=v(\varpi^ku)=k.$$
Therefore, $j=2m-k$ and by Proposition~\ref{p:GeneralBound} we have
$$J_0(\phi^{(0,m)},t) \ll (m+v(2)-k)^5q^{-4m+\frac{k}2}.$$
Summing over $k<0$ and $u \in (\OF/\p^m)^\times$ and multiplying by $\vol(\OF^\times \bs U_T(m))$
gives a contribution  $\ll (m+v(2))^5q^{-4m-\frac{1}2}.$

Consider now the case $k>0$. Then $v(y)=0$ and, by the same reasoning as as above,
$v(x\pm  y\frac{b}2) \ge 0$, hence $$i=-m$$ and $$2i+j=k.$$
Therefore, $j=2m+k$, and from then on this case is completely analogous to the case $k<0$ after changing $k$ to $-k$.

Finally, consider the case $k=0$. We again have $v(x\pm  y\frac{b}2) \ge 0$.
Since we can always choose the representative $u \in (\OF/ \p^m)^\times$
such that $v(y) \le m$, it follows that $i=v(y)-m$ and $2i+j=0$. Therefore by Proposition~\ref{p:GeneralBound} we have
$$J_0(\phi^{(0,m)},t ) \ll (m+v(2))^5q^{-4m+v(y)}.$$
Since there are $\asymp q^{m-v(y)}$ such representatives, we get as before a contribution
$\ll m(m+v(2))^5 q^{-4m}.$

Adding up the contributions of all the representatives thus gives
$$ J_{\Lambda,\theta}(\phi^{(0,m)}) \ll m(m+v(2))^5 q^{-4m}.$$

\section{Bounds on Fourier coefficients and sup-norms of Siegel cusp forms} In this section we relate the Fourier coefficients of Siegel cusp forms of degree 2 to global Bessel coefficients and the local quantities studied in the previous section. We go on to prove the main theorems stated in the introduction.
\subsection{Siegel cusp forms and representations}
We begin with recalling classical Siegel cusp forms and their adelizations. For brevity, denote $G(R):=\GSp_4(R)$ for each ring $R$ and let $\Gamma = \Sp_4(\Z)$. Let $k$ be a positive integer.

Let $S_k(\Gamma)$ be the space of holomorphic Siegel cusp forms of degree $2$ and weight $k$ with respect to $\Gamma$. Hence, if $F \in S_k(\Gamma)$, then for all $\gamma \in \Gamma$ we have $F |_k \gamma=F$, where
\begin{equation}\label{Siegel mod form}
    (F |_k g)(Z) := \mu(g)^k j(g, Z)^{-k}
    F(g \langle Z \rangle)
\end{equation}
for $g \in G(\R)^+$ and $Z \in \H_2$, the Siegel upper half space; moreover $F$ vanishes at the cusps. For a precise formulation of this cusp vanishing condition, see~\cite[I.4.6]{Fr1991};
 for definitions and basic properties
of Siegel cusp forms we refer the reader to \cite{andzhu}.

For $F \in S_k(\Gamma)$, we define the adelization $\phi_F$ of $F$ to be the function on
$G(\A)$ defined by
\begin{equation} 
\phi_F(\gamma h_\infty k_0) =
  \mu(h_\infty)^k j(h_\infty,
  iI_2)^{-k}F(h_\infty \langle iI_2\rangle)
\end{equation}
where $\gamma \in G(\Q), h_\infty \in G(\R)^+$ and
$k_0 \in \prod_{ p<\infty } G(\Z_p)$. Then $\phi_F$ is a
well-defined function on the whole of $G(\A)$ by strong approximation,
and is a cuspidal automorphic form.

It is well-known that $\phi_F$ generates an irreducible representation  if and only if $F$ is an eigenform of all the Hecke operators \cite{NPS}. We denote this representation by $\pi_F$. Now suppose that $F$ is a Hecke eigenform. Since $F$ is of full level, it is either of general type or of Saito--Kurokawa type (see, e.g., Proposition 2.3.1 of \cite{ps22}). The representation $\pi_F$ is an irreducible cuspidal automorphic representation of $G(\A)$ of trivial central character; writing $\pi_F = \otimes_v \pi_v$, we have
\begin{itemize}
 \item The archimedean component $\pi_\infty$ is a holomorphic discrete series representation with scalar minimal $K$-type determined by the weight $k$.
 \item If $F$ is of Saito--Kurokawa type, then for a prime number $p$, the representation $\pi_p$ is of type IIb according to Table A.1 of \cite{NF}. Note that these are non-tempered, non-generic representations.
 \item If $F$ is of general type, then for a prime number $p$, $\pi_p$ is a \emph{tempered} representation of Type I in the notation of \cite{NF}.
\end{itemize}

For $F \in S_k(\Gamma)$ we define the Petersson norm \begin{equation}\label{eqn:petersson-def}
\|F\|_2^2 = \langle F, F\rangle
=
\int\limits_{\Gamma \bs \H_2} |F(Z)|^2 (\det Y)^{k - 3}\,dX\,dY.
\end{equation}
For any $\phi \in L^2(\A^\times G(\Q) \bs G(\A))$, let $\langle \phi, \phi\rangle = \int_{\A^\times G(\Q) \bs G(\A)} |\phi(g)|^2\,dg$. We have the convenient relation \begin{equation}\label{e:petersson}
  \frac{\langle F, F\rangle}{\vol(\Sp(4,\Z)\bs \H_2)} = \frac{\langle \phi_F, \phi_F \rangle}{\vol(Z(\A)G(\Q) \bs G(\A))};
\end{equation}
c.f. Section 4.1 of \cite{asgsch}.

\subsection{Bessel periods and Fourier coefficients}\label{s:besselfour} Given a fundamental discriminant $d<0$, define
\begin{equation}\label{e:defS}
S_d=  \begin{cases} \left[\begin{smallmatrix}
  \frac{-d}{4} & 0\\
 0 & 1\\\end{smallmatrix}\right] & \text{ if } d\equiv 0\pmod{4}, \\[2ex]
 \left[\begin{smallmatrix} \frac{1-d}{4} & \frac12\\\frac12 & 1\\
 \end{smallmatrix}\right] & \text{ if } d\equiv 1\pmod{4}.\end{cases}
\end{equation} Given $S_d$ as above, let the group $T_d=T_{S_d}$ over $\Q$ be defined in the same way as in \eqref{TFdefeq}.  So $T_d\simeq K^\times$ where $K=\Q(\sqrt{d})$. Note that the matrix $S_d$ above satisfies the standard assumptions \eqref{standardassumptions} at every finite prime.

Let $\psi:\Q\backslash \A \rightarrow \C^\times$ be the character such that $\psi(x) = e^{2 \pi i x}$ if $x \in \R$ and $\psi(x) = 1$ for $x \in \Z_p$.  One obtains a character $\theta_{S_d}$ of $N(\Q) \backslash N(\A)$ by $\theta_{S_d}(\mat{1}{X}{}{1}) = \psi({\rm Tr}(S_d X))$. Let $\Lambda$ be a character of $K^\times \bs \A_K^\times$ such that $\Lambda|_{\A^\times} = 1$. Then for a measurable function $\phi: \A^\times G(\Q) \bs G(\A) \rightarrow \C$, we define the Bessel period
\begin{equation}\label{defbesselnew}
  B(\phi, \Lambda) =
  \int\limits_{\A^\times T_d(\Q)\bs T_d(\A)}\;\int\limits_{N(\Q) \bs N(\A)}\phi(tn)\Lambda^{-1}(t) \theta_{S_d}^{-1}(n)\,dn\,dt,
\end{equation}
where we give the adelic groups the Tamagawa measure.

In the special case where $\phi=\phi_F$ with $F$ a Siegel cusp form of degree 2, the function  $B(\phi_F, \Lambda)$ captures information about all the Fourier coefficients of $F$. Let us now make this precise.
Let $F \in S_k(\Gamma)$. Then $F$ has a Fourier expansion \begin{equation}\label{siegelfourierexpansion}F(Z)
=\sum_{T \in \Lambda_2} a(F, T) e^{2 \pi i \Tr(TZ)}
\end{equation}
with \begin{equation}\label{e:Lambda2} \tst
 \Lambda_2 = \left\{\mat{a}{b/2}{b/2}{c}:\qquad a,b,c\in\Z, \qquad a>0, \qquad  d:=b^2 - 4ac <0\right\}.
\end{equation} 
For a matrix $T = \mat{a}{b/2}{b/2}{c}\in \Lambda_2$, we distinguish its content $c(T)=\gcd(a,b,c)$ and discriminant $\disc(T)=-4\det(T)$. Clearly, $c(T)^2$ divides $\disc(T)$.

From the invariance property of $F$, we easily see that its Fourier coefficients satisfy the relation
\begin{equation}\label{fourierinvariance}a(F, T) = \,a(F, \T{A}TA) \end{equation} for any $A \in \SL_2(\Z)$.
Therefore, for $T_1, T_2 \in \Lambda_2$, we consider an equivalence relation
$$T_1\sim T_2\qquad \Longleftrightarrow\qquad \text{there exists } A \in\SL_2(\Z) \text{ such that }\; ^t\!A T_1 A = T_2$$
and say that the matrices $T_1, T_2$ satisfying this relation are $\SL_2(\Z)$-equivalent. The equivalence class of $T$ will be denoted by $[T]$. Note that:
\begin{itemize}
\item the Fourier coefficient $a(F, T)$ depends only on the $\SL_2(\Z)$-equivalence class of $T$;
\item if $T_1\sim T_2$, then $\disc (T_1)=\disc (T_2)$ and $c(T_1)=c(T_2)$.
\end{itemize}

Given any $T \in \Lambda_2$, it is clear that there exists a fundamental discriminant $d<0$ and  positive integers $L$, $M$ such that $c(T)=L$, $\disc(T) = dL^2M^2$; furthermore $d, L,M$ are uniquely determined by $T$.
For a fundamental discriminant $d < 0$ and positive integers $L$, $M$, we let $H(dM^2; L)$ denote the set of $\SL_2(\Z)$-equivalence classes of matrices in $\Lambda_2$ such that $c(T)=L$ and $\disc(T) = dL^2M^2$. It is clear that given $[X] \in H(dM^2; L)$ and $F \in S_k(\Gamma)$, the notation $a(F, [X])$ is well-defined.

We define
$$
 \Cl_d(M) = T_d(\A) /T_d(\Q)T_d(\R)\prod_{p<\infty}U_{T_d}(m_p)
$$
where we write $M= \prod_{p<\infty}p^{m_p}$ and the subgroup $U_{T_d}(m_p) \subset T_d(\Z_p)$ is defined in Section \ref{s:levelstructure}.
It is well known that  $\Cl_d(M)$ can be naturally identified with the class group of the unique order of discriminant $D=dM^2$ in $\Q(\sqrt{d})$; in particular, $\Cl_d(1)$ is canonically isomorphic to the ideal class group of $\Q(\sqrt{d})$. From Proposition 5.3 of \cite{pssmb} we have
$
 |\Cl_d(M)| = \frac{M}{u(d)}\,|\Cl_d(1)|\,\prod_{p|M} \left( 1 -  \Big(\frac{d}p\Big) p^{-1} \right)$ where $u(-3)=3$, $u(-4)=2$ and $u(d)=1$ for other $d$. In particular this shows that \begin{equation}\label{e:boundcld}(|d|^{1/2}M)^{1-\eps} \ll_\eps |\Cl_d(M)| \ll_\eps (|d|^{1/2}M)^{1+\eps}.\end{equation}

Section 5 of \cite{pssmb} constructs for each $c \in T_d(\A)$ a matrix $\phi_{L,M}(c) \in \Lambda_2$ such that $\phi_{L,M}(c) =\mat{L}{}{}{L}\mat{M}{}{}{1}S_c\mat{M}{}{}{1}$ with $\disc(S_c) = d$ and the $(2,2)$-coefficient of $S_c$ is $1$ modulo $M$; this implies that $$c(\phi_{L,M}(c))=L, \quad \disc(\phi_{L,M}(c))=dL^2M^2.$$ The matrix $\phi_{L,M}(c)$ depends on some choices, but its class in $H(dM^2;L)$ is independent of those choices. We recall a result from \cite{pssmb}.
\begin{lemma}[Prop 5.3 of \cite{pssmb}]\label{l:classidentification}For each pair of positive integers $L,M$, the map $c \mapsto [\phi_{L,M}(c)]$ gives a bijection from $\Cl_d(M)$ to $H(dM^2; L)$.
\end{lemma}

Write $L=\prod_{p<\infty}p^{\ell_p}$, $M=\prod_{p<\infty} p^{m_p}$. Define the element $H(L,M)\in G(\A)$ via
$$H(L,M)_p:=\begin{cases} h_p(\ell_p, m_p) ,& p\mid LM\\ 1 ,& p\nmid LM\mbox{ or } p=\infty\end{cases}\,,$$ where the element $h_p(\ell_p, m_p)$ was defined in \eqref{hlmdefeq}; note that we now add the subscript $p$ to avoid confusion.  For an automorphic form $\phi$ on $G(\A)$, we let $\phi^{L,M}$ be the automorphic form obtained by right-translation by $H(L,M)$, i.e., $\phi^{L,M}(g):=\phi(gH(L,M)),$ for all $g \in G(\A)$. We can now state the relation between Bessel coefficients and Fourier coefficients.
\begin{lemma}\label{lemma:relation besselfour}Let $F \in S_k(\Gamma)$ and $\phi_F$ be its adelization. For any character $\Lambda$ of $\Cl_d(M)$, we have
\begin{align*}
B(\phi_F^{L,M},\Lambda)=\frac{2}{|\Cl_d(M)|}(LM)^{-k}e^{-2\pi\Tr (S_d)}\sum_{c\in\Cl_d(M)} \Lambda^{-1}(c) a(F,\phi_{L,M}(c)).
\end{align*}
\end{lemma}
\begin{proof}This is a standard calculation; see, e.g., the proof of \cite[Theorem 3]{marzec21}.
\end{proof}

Applying orthogonality relations and noting the bijection between $\Cl_d(M)$ and $H(dM^2; L)$, we obtain the key formula
\begin{align}\label{coeff-formula}
\sum_{[T] \in H(dM^2; L)}|a(F,[T])|^2=& \frac{1}{4}(LM)^{2k}e^{4\pi\Tr (S_d)} |\Cl_d(M)|\sum_{\Lambda\in\widehat{\Cl_d(M)}} |B(\phi_F^{L,M},\Lambda)|^2.
\end{align}

\subsection{The refined GGP identity and bounds on Fourier coefficients}
In the case that $F \in S_k(\Gamma)$ is a Hecke eigenform, we can refine the formula \eqref{coeff-formula}. In this case, the adelization of $F$ generates an irreducible automorphic representation $\pi_F$. Let $K=\Q(\sqrt{d})$ and $\Lambda$ be a character of $K^\times \bs \A_K^\times$ such that $\Lambda|_{\A^\times} = 1$. Liu \cite{yifengliu} formulated a precise refinement of the Gan--Gross--Prasad conjecture  for $(\SO(5), \SO(2))$ which applies in particular to the Bessel periods $B(\phi, \Lambda)$ where $\phi$ is any automorphic form in the space of $\pi_F$. Recently Furusawa and Morimoto have proved this conjecture \cite{FM22} for tempered representations; in particular their result applies to  $\pi_F$ whenever $F$ is not of Saito--Kurokawa type (the Saito--Kurokawa case is easier and can be dealt with separately; see, e.g., \cite{qiu}).  \begin{theorem}[Furusawa--Morimoto, Theorem 1.2 of \cite{FM22}]\label{c:liu}
 Let $\pi_F = \otimes_v \pi_v$ and $\Lambda$ be as above and assume that $\pi_F$ is tempered. Let $\phi=\otimes_v \phi_v$ be an automorphic form in the space of $\pi_F$. Let $S$ be a set of places including the archimedean place such that $\pi_v, K_v$, $\phi_v$, and $\Lambda_v$  are all unramified for $v \notin S$. Then
 \begin{equation}\label{e:refggpnew}
  \frac{|B(\phi, \Lambda)|^2}{\langle \phi, \phi \rangle} = \frac{C_T}{S_{\pi_F}}\frac{\xi(2)\xi(4)L^S(1/2, \pi_F \times \AI(\Lambda^{-1}))}{L^S(1, \pi_F, \Ad)L^S(1, \chi_d)} \prod_{v\in S} J_{\Lambda_v, \theta_v}(\phi_v),
 \end{equation}
 where $\xi(s) = \pi^{-s/2}\Gamma(s/2)\zeta(s)$ denotes the completed Riemann zeta function, $C_T$ is a constant relating our choice of local and global Haar measures, $S_{\pi_F}$ denotes a certain integral power of 2, related to the Arthur parameter of $\pi_F$, and  $J_{\Lambda_v, \theta_v}(\phi_v)$ equals the local Bessel integral defined in Section \ref{s:besselmodels}.
\end{theorem}

\begin{proposition}\label{p:main}Let $F \in S_k(\Gamma)$ and let $\phi_F$ be its adelization. Assume that $F$ is not a Saito--Kurokawa lift and let $\pi_F$ be the automorphic representation generated by $F$.  Let $L$, $M$ be positive integers and let $\Lambda$ be a character of $\Cl_d(M)$. Then
\[\frac{|B(\phi_F^{L,M},\Lambda)|^2}{\langle \phi_F, \phi_F\rangle} \ll_\eps   \frac{(L Md)^\eps}{M^4L^{3}} \frac{(4    \pi)^{2k}  |d|^{k-2} e^{-4\pi {\rm Tr}(S_d)}}{\Gamma(2k-1)} \frac{L(1/2, \pi_F \times \AI(\Lambda^{-1}))}{L(1, \pi_F, \Ad)}.\]
\end{proposition}
\begin{proof}
 We return to the setup of Section \ref{s:besselfour} but assume also that $F$ is a Hecke eigenform. Write $\phi_F = \otimes_v \phi_v$, $\pi_F =\otimes_v \pi_v$ and $\phi_F^{L,M} = \otimes_v \phi_v^{(l_v,m_v)}$. By the uniqueness of local and global Bessel functions, we have the following relation (see also Lemma 5 in \cite{marzec21}).
\begin{equation}\label{e:keyfactorizationbessel}B(\phi_F^{L,M},\Lambda) = \(\prod_{p|LM} \frac{B_{\phi_p, \Lambda_p}(h_p(\ell_p, m_p))}{B_{\phi_p, \Lambda_p}(h_p(0, c(\Lambda_p)))}\) B(\phi_F^{1,C(\Lambda)},\Lambda)
\end{equation} where $C(\Lambda) = \prod_{p|M} p^{c(\Lambda_p)}$.
Using Proposition \ref{p:localboundbesselsugano}, we therefore obtain
\begin{equation}\label{e:besselreduction}|B(\phi_F^{L,M},\Lambda)|^2 \ll_\eps (L M)^\eps \frac{C(\Lambda)^4}{M^4L^{3}} \ |B(\phi_F^{1,C(\Lambda)},\Lambda)|^2.\end{equation}
On the other hand, using \eqref{e:refggpnew}, Proposition \ref{p:localboundJp}, and the computation $C_TJ_\infty \asymp \frac{(4    \pi)^{2k}  |d|^{k-2} e^{-4\pi {\rm Tr}(S_d)}}{\Gamma(2k-1)L(1, \chi_d)}$ which follows from \cite[Sec. 3.5]{DPSS15} we have \begin{equation}\label{e:keyboundbessel2}\frac{|B(\phi_F^{1,C(\Lambda)},\Lambda)|^2}{\langle \phi_F, \phi_F\rangle} \ll_\eps C(\Lambda)^{-4 + \eps} \frac{(4    \pi)^{2k}  |d|^{k-2} e^{-4\pi {\rm Tr}(S_d)}}{\Gamma(2k-1)} \frac{L(1/2, \pi_F \times \AI(\Lambda^{-1}))}{L(1, \pi_F, \Ad)L(1, \chi_d)^2}\end{equation}
The proof follows by combining \eqref{e:besselreduction}, \eqref{e:keyboundbessel2} and the well-known fact $L(1, \chi_d) \gg_\eps |d|^{-\eps}$.
\end{proof}
We are ready to prove our main theorem.
\begin{theorem}\label{t:mainfouriergen}Let $F \in S_k(\Gamma)$ be a Hecke eigenform that is not a Saito--Kurokawa lift and let $\pi_F$ be the automorphic representation generated by $F$.  Let $d<0$ be a fundamental discriminant and let $L$, $M$ be positive integers.  \begin{enumerate}
\item For each character $\Lambda$ of $\Cl_d(M)$, we have \[\left|\sum_{c\in\Cl_d(M)} \Lambda(c) a(F,\phi_{L,M}(c)) \right|^2 \ll_\eps  \langle F, F \rangle \frac{(4    \pi)^{2k}  L^{2k-3+\eps}|dM^2|^{k-1+\eps} }{\Gamma(2k-1)} \frac{L(1/2, \pi_F \times \AI(\Lambda))}{L(1, \pi_F, \Ad)}.  \]

\item \label{e:second} We have the bound
\[\sum_{[T] \in H(dM^2; L)}|a(F,[T])|^2 \ll_\eps \langle F, F \rangle \frac{(4 \pi)^{2k}}{\Gamma(2k-1)} L^{2k-3 + \eps}|dM^2|^{k - \frac{3}{2} + \eps} \sum_{\Lambda\in\widehat{\Cl_d(M)}} \frac{L(1/2, \pi_F \times \AI(\Lambda))}{L(1, \pi_F, \Ad)}.\]
\end{enumerate}
\end{theorem}
\begin{proof}The first assertion follows from \eqref{e:petersson}, \eqref{e:boundcld}, Lemma \ref{lemma:relation besselfour} and Proposition \ref{p:main}. The second assertion follows from \eqref{e:petersson}, \eqref{e:boundcld}, \eqref{coeff-formula} and Proposition \ref{p:main}.
\end{proof}
\begin{remark}Note that Theorem \ref{t:mainfourier} is just a restatement of part \eqref{e:second} of Theorem \ref{t:mainfouriergen}. In the terminology of Theorem \ref{t:mainfourier} we have $D=dM^2$, $H_D = \Cl_d(M)$.
\end{remark}
\begin{corollary}\label{FourierCoeffBound}
Let $F \in S_k(\Gamma)$ be a Hecke eigenform that is not a Saito--Kurokawa lift. Let $d<0$ be a fundamental discriminant and $M$ a positive integer.  Assume that for some real numbers $\alpha, \beta$, the bound
    \begin{equation}\label{e:reqboundGRH}\frac{L(1/2, \pi_F \times \AI(\Lambda))}{L(1, \pi_F, \Ad)}\ll_\epsilon k^{\alpha+\epsilon} (|d|M^2)^{\beta+\epsilon}\end{equation}
holds for all $\Lambda \in \widehat{\Cl_d(M)}$. Then
    $$\frac1{\langle F,F \rangle}
\sum_{[T] \in H(dM^2; L)} |a(F,[T])|^2
\ll_\eps    \frac{k^{1/2 + \alpha + \eps} (2\pi)^{2k}}{\Gamma(k)^2} L^{-1-2\beta} |dM^2L^2|^{k-1 + \beta + \eps}.$$ Furthermore, for any $T \in \Lambda_2$, we have $$\frac{|a(F,T)|}{\|F\|_2} \ll_{\beta,\eps}  \frac{k^{\frac14 + \frac{\alpha}{2}+\eps} (4\pi)^k}{\Gamma(k)} c(T)^{-\frac12 - \beta} \det(T)^{\frac{k-1+\beta}2+\eps}.$$
\end{corollary}
\begin{proof}
The first assertion follows by substituting the bound \eqref{e:reqboundGRH} into Theorem \ref{t:mainfourier} and using the duplication formula for the Gamma function. For the second assertion, we note that for any $T \in \Lambda_2$ we have $[T] \in H(dM^2; L)$ where we put $L = c(T)$, $dM^2 = -\frac{4\det(T)}{c(T)^2}$; now the second assertion follows from the first by dropping all but one term.
\end{proof}

\begin{remark}
   For $\Lambda \in \widehat{\Cl_d(M)}$, the analytic conductor of $\AI(\Lambda)$ is bounded by $|d|M^2$ \cite[(a2)]{Rohrlich1994}. As the analytic conductor of $\pi_F$ is $\asymp k^2$, the convexity bound gives $L(1/2, \pi_F \times \AI(\Lambda))  \ll_\eps  (k|d|M^2)^{1+\eps}$ and the  Generalized Lindel\"of hypothesis asserts that   $L(1/2, \pi_F \times \AI(\Lambda))  \ll_\eps  (k|d|M^2)^{\eps}.$  On the other hand, lower bounds for $L(1, \pi_F, \Ad)$ follow from good zero-free regions. In particular, it is known unconditionally (use \cite[Theorem 3]{brumley06} and functoriality) that there exists an absolute constant $A$ such that  $L(1, \pi_F, \Ad) \gg k^{-A}$; under GRH we have $L(1, \pi_F, \Ad) \gg_\eps k^{-\eps}$. It follows from the above discussion that under GRH we may take $\alpha = \beta = 0$ in \eqref{e:reqboundGRH}.
\end{remark}

\subsection{Sup-norms of Siegel cusp forms}

\subsubsection{Preparatory lemmas}
The results in this subsection are implicitly or explicitly contained in~\cite{Blomer}.
Let $Y=\mat{x}{\frac{y}2}{\frac{y}2}{z} \in \Mat_2(\R)$. We say that $Y$ is Minkowski-reduced if the following holds
\begin{alignat*}{3}
|y| \le x \le z, &\qquad x \gg 1, &\qquad \det Y \asymp xz.
\end{alignat*}
Let $T = \mat{a}{\frac{b}2}{\frac{b}2}{c} \in \Lambda_2$. Then $TY$ is diagonalizable over $\R$ with positive eigenvalues; we shall denote them by $x_1,x_2$.
We need some control over the counting function
$$N_Y(\lambda,h)=\# \{ T \in \Lambda_2 : \lambda-h \le  x_1, x_2 \le \lambda+h\}.$$
The first lemma provides a trivial bound.

\begin{lemma}\label{trivialcount}
   Assume that $Y$ is Minkowski-reduced.
   Then for $\lambda>0$ and $h>1$ we have
     $$N_Y(\lambda,h) \ll \frac{h(\lambda+h)^2}{(\det Y)^{\frac32}}.$$
\end{lemma}
\begin{proof} For convenience, define $m=\lambda-h$ and $M=\lambda+h$.
        Then $x_1,x_2$ are the roots of the polynomial $X^2 - \Tr(TY)X+\det(TY)$.
    In particular, if $m \le x_1, x_2 \le M$ then it follows
    \begin{equation}\label{traceinequality}
        2m \le x_1+x_2= \Tr(TY)=ax+\frac12by+cz \le 2M
    \end{equation}
    and
    \begin{equation}\label{determinantinequality}
        m^2 \le x_1x_2 = \det(TY) \le M^2.
    \end{equation}
    From~(\ref{traceinequality}) we obtain
    \begin{equation}\label{rangex}
        2m - \frac12by-cz \le ax \le 2M-\frac12by-cz.
    \end{equation}
    Since $c>0$, substituting~(\ref{rangex}) in~(\ref{determinantinequality}),
   and using that $Y$ is Minkowski-reduced, we get
    $$c\left(2M-\frac{b}2y-cz\right)-\frac{x}4b^2 \ge {\det(T)x} \ge \frac{m^2}z.$$
    Changing variables $\tilde{c}=c-\frac{xM}{\det Y}$ and $\tilde{b}=b+\frac{yM}{\det Y}$ we obtain
    \begin{equation}\label{ellipse}
z\tilde{c}^2+\frac{x}4\tilde{b}^2+\frac{y}2\tilde{b}\tilde{c} \ll \frac{xM^2}{ \det Y}-\frac{m^2}{z} \ll \frac{M^2}{z}.
    \end{equation}
    Completing the square,
    $$z\left(\tilde{c}+\frac{y}{4z}\tilde{b}\right)^2+\frac{\det Y}{4z}\tilde{b}^2 \ll \frac{M^2}{z}$$
    and thus $|\tilde{c}+\frac{y}{4z}\tilde{b}| \ll \frac{M}{z}$,
    $|\tilde{b}| \ll \frac{M}{\sqrt{\det(Y)}}$. So there are $\ll \frac{M^2}{z \sqrt{\det Y}}$ choices for the pair $(b,c)$.
    For each choice of $(b,c)$, by~(\ref{rangex}) there are $\ll \frac{M-m}{x}$ choices for
    $a$, giving the desired bound.
\end{proof}
\begin{lemma}[Blomer, Lemma 4 of \cite{Blomer}]\label{Blomercount}
    Assume that $Y$ is Minkowski-reduced and that \linebreak $h=O(\sqrt{k}\log{k})$.
    Then $$N_Y\left(\frac{k}{4\pi},h\right) \ll \frac{k^{\frac32+\epsilon}}{(\det Y)^{\frac34}}.$$
\end{lemma}
For $x > 0$ and $k > 0$ define
$$g_k(x)=x^{\frac{k}2}\exp(-2\pi x), \quad x_k=\frac{k}{4 \pi}.$$
\begin{lemma}\label{realanalysis}
The function $g_k$ reaches its maximum at $x_k$, and for $k \gg 1$ we have
    \begin{enumerate}
        \item if $|x-x_k| \ge \sqrt{k} \log(k)$ then $g_k(x) \ll k^{-100} g_k(x_k)$,
        \item if $x \ge 2x_k$ then $g_k(x) \le \exp(-\frac{\pi}3(x-x_k))g_k(x_k)$.
    \end{enumerate}
\end{lemma}
\begin{proof}
    The derivative of $\log(g_k)$ is given by
    $\frac{d}{dx} \log g_k(x) = 2\pi (\frac{x_k}x-1),$
    thus $g_k$ increases from $0$ to $x_k$ and decreases after this point, which
     establishes the first claim.
     To show~(1), it suffices to bound $g_k(x_k \pm \sqrt{k}\log(k))$.
    If $x=x_k-\sqrt{k}\log(k)$ then we have
    \begin{align*}
        \frac{g_k(x)}{g_k(x_k)}=\exp\left(
        \frac{k}2\log\left(1-4\pi \frac{\log (k)}{\sqrt{k}}\right)
        +2\pi \sqrt{k} \log(k)\right)\\
        \le \exp\left(-4\pi^2(\log k)^2\right).
    \end{align*}
    On the other hand if $x \ge x_k$ then we have
    \begin{align*}
        \log g_k(x)&= \log\left(g_k\left(\frac{x+x_k}{2}\right)\right)+2\pi\int_{\frac{x+x_k}2}^x \left(\frac{x_k}t-1\right) \, dt\\
        & \le \log(g_k(x_k))+2\pi\int_{\frac{x+x_k}2}^x \frac{x_k-x}{x+x_k} \, dt\\
        & = \log(g_k(x_k))-\pi\frac{(x-x_k)^2}{x+x_k}.
    \end{align*}
    Thus
    \begin{equation}\label{boundgk}
        g_k(x) \le g_k(x_k)\exp\left(-\pi \frac{(x-x_k)^2}{x+x_k}\right).
    \end{equation}
    In particular, for $x=x_k+\sqrt{k}\log(k)$, by~(\ref{boundgk}) and using the fact that $\sqrt{k}\log(k) \le \frac2ek$, we have
    $$g_k(x) \le g_k(x_k) \exp\left(-\frac{2\pi^2e}{e+4\pi}(\log k)^2\right).$$
    This proves~(1).
    Finally, if $x \ge 2x_k$ then $\frac{x-x_k}{x+x_k} \ge \frac13$  hence~(\ref{boundgk})
    gives~(2).
\end{proof}

\subsubsection{Main result}
\begin{theorem}\label{t:mainsup}
       Let $F \in S_k(\Gamma)$ be a Hecke eigenform that is not a Saito--Kurokawa lift and normalized so that $\|F\|_2=1$.
    Let $0\le \alpha \le 2$ and $0\le \beta \le1$, be constants such that the bound \eqref{e:reqboundGRH} holds for all $\Lambda \in \widehat{\Cl_d(M)}$ for all fundamental discriminants $d$ and positive integers $M$.
    Then for all $Z=X+iY \in \H_2$ we have  $$|(\det Y)^{\frac{k}2}F(Z)| \ll_\epsilon k^{\frac54+\beta+\frac{\alpha}2+\epsilon}. $$
\end{theorem}
\begin{proof}
    Write the Fourier expansion of $F$ and apply
    Corollary~\ref{FourierCoeffBound}:
    \begin{align*}
        (\det Y)^{\frac{k}2}F(Z) &= (\det Y)^{\frac{k}2}\sum_{T \in \Lambda_2} a(F, T) e^{2 \pi i \Tr(TZ)}\\
        & \ll_\epsilon \frac{(4\pi)^k}{\Gamma(k)} k^{\frac14+\frac{\alpha}2+\epsilon}\sum_{T \in \Lambda_2}
        \frac{c(T)^{-\frac12-\beta}}{\det(T)^{\frac12-\frac\beta2-\epsilon}}\det(TY)^{\frac{k}2}e^{-2 \pi  \Tr(TY)} .
    \end{align*}
    By Stirling's formula, we have
    \begin{equation}\label{Stirling}
        \frac{(4\pi)^k}{\Gamma(k)} \asymp k^{\frac12}
        \left(\frac{4\pi e}{k}\right)^k .
    \end{equation}
    We then proceed as in~\cite{Blomer}: We may assume without loss of generality that $Y$ is
    Minkowski-reduced, and furthermore the matrix $X=TY$ is diagonalizable
    with positive eigenvalues $x_1,x_2$, and  $\det(X)^{\frac{k}2}e^{-2\pi\Tr(X)}=g_k(x_1)g_k(x_2)$.
    Let $\mathcal X_0$ be the set of matrices $T \in \Lambda_2$ such that
    $\frac{k}{4\pi}-\sqrt{k}\log(k) \le x_1, x_2 \le \frac{k}{4\pi}+\sqrt{k}\log(k)$.
    For those matrices $T \in \mathcal X_0$, we have $\det(T) \asymp k^2 (\det Y)^{-1}$ and thus
    by~(\ref{Stirling}), using Lemma~\ref{realanalysis} and bounding $c(T)\ge 1$ trivially we have
    $$\frac{(4\pi)^k}{\Gamma(k)}
    k^{\frac14+\frac{\alpha}2+\epsilon}
    \frac{c(T)^{-\frac12-\beta}}
    {\det(T)^{\frac12-\frac\beta2-\epsilon}}
    \det(TY)^{\frac{k}2}e^{-2 \pi  \Tr(TY)}
    \ll (\det Y)^{\frac12-\frac\beta2}
    k^{-\frac14+\beta+\frac{\alpha}2+\epsilon}.$$
    Furthermore, by Lemma~\ref{Blomercount} we have
    $\# \mathcal X_0 \ll k^{\frac32+\epsilon}(\det Y)^{-\frac34}$ such matrices $T$,
    and thus we get a contribution
    $$(\det Y)^{\frac{k}2}\sum_{T \in \mathcal X_0} a(F, T) e^{2 \pi i \Tr(TZ)}
    \ll_\epsilon (\det Y)^{-\frac14-\frac\beta2}k^{\frac54+\beta+\frac{\alpha}2+\epsilon}.$$
    Now for $j > 0$,  let $\mathcal X_j$ be the set of matrices $T \in \Lambda_2$ such that $0<x_1,x_2 \le (j+1)\frac{k}{4\pi}$ and $T \not \in \bigcup_{i=0}^{j-1} \mathcal X_i$.
    Then combining Lemma~\ref{trivialcount}, part~(1) of Lemma~\ref{realanalysis}
    and~(\ref{Stirling}) we get
    \begin{align*}
        (\det Y)^{\frac{k}2}\sum_{T \in \mathcal X_1} a(F, T) e^{2 \pi i \Tr(TZ)}
    &
    \ll k^{-50}(\det Y)^{-\frac32}.
    \end{align*}
    Finally, for $j \ge 2$, combining Lemma~\ref{trivialcount}, part~(2) of Lemma~\ref{realanalysis}
    and~(\ref{Stirling}) we get
    \begin{align*}
        (\det Y)^{\frac{k}2}\sum_{T \in \mathcal X_j} a(F, T) e^{2 \pi i \Tr(TZ)}
    \ll \exp\left(-\frac{(j-1)k}{13}\right)(\det Y)^{-\frac32},
    \end{align*}
    and thus
    $$ (\det Y)^{\frac{k}2}F(Z) = (\det Y)^{\frac{k}2}\sum_{j \ge 0}\sum_{T \in \mathcal X_j} a(F, T) e^{2 \pi i \Tr(TZ)} \ll_\epsilon (\det Y)^{-\frac14-\frac\beta2-\epsilon}k^{\frac54+\beta+\frac{\alpha}2+\epsilon} \ll k^{\frac54+\beta+\frac{\alpha}2+\epsilon}.$$
\end{proof}

\bibliography{ggp-supnorm}{}
\bibliographystyle{alpha}

\end{document}